\documentclass[reqno,12 pt]{amsart}

\usepackage{amsfonts,amscd}
\usepackage{amsthm}
\usepackage{amssymb}
\usepackage{enumerate}
\usepackage{latexsym}
\usepackage{wasysym}
\usepackage{stmaryrd}
\usepackage{color}
\usepackage{url}
\usepackage{cite}
\usepackage{amsmath}
\usepackage{verbatim}
\usepackage{graphicx}
\usepackage{lineno}
\usepackage{xcolor}
\usepackage{multirow}
\usepackage{amssymb}

\usepackage{algorithm}
\usepackage{algorithmicx}
\usepackage{algpseudocode}
\usepackage{epsf}
\usepackage{geometry}
\usepackage{hyperref}
\definecolor{dark-blue}{rgb}{0,0,0.6}
\definecolor{Purple}{rgb}{0.2,0,0.25}
\hypersetup{breaklinks=true,
pagecolor=white,
colorlinks=true,
citecolor=dark-blue,
filecolor=black,%
linkcolor=Purple,%
urlcolor=blue
}

\theoremstyle{plain} 
\newtheorem{Thm}{Theorem}[section]
\newtheorem{Lem}[Thm]{Lemma}

\newtheorem{proposition}[Thm]{Proposition}
\newtheorem{problem}[Thm]{Problem}

\theoremstyle{definition}
\newtheorem{remark}[Thm]{Remark}
\newtheorem{defin}[Thm]{Definition}
\newtheorem{alg}[Thm]{Algorithm}

\numberwithin{equation}{section}

\newcommand{\wt}{\widetilde}

\newcommand{\Int}{\textnormal{int}}

\newcommand{\Id}{\textnormal{Id}}

\newcommand{\R}{\mathbb{R}}

\newcommand{\N}{\mathbb{N}}

\newcommand{\wh}{\widehat}

\newcommand{\dist}{\textnormal{dist}}

\newcommand{\bref}[1]{\textbf{\ref{#1}}} 
\newcommand{\beqref}[1]{\textbf{(\ref{#1})}} 

\begin{document}


\date{October 8, 2026}


\title[The best approximation tuple]{The best approximation tuple: an extension of the Cheney-Goldstein algorithm and results to the multiple sets case}

\author{Yair Censor}
\address{Department of Mathematics, University of Haifa, Mt. Carmel, Haifa 3498838, Israel}
\email{yair@math.haifa.ac.il}

\author{Daniel Reem}
\address{The Center for Mathematics and Scientific Computation (CMSC), University of Haifa, Mt. Carmel, Haifa 3498838, Israel}
\email{dream@math.haifa.ac.il}


\maketitle

\begin{abstract}
We extend the algorithm and several results published in the celebrated 1959 paper of Cheney and Goldstein about the best approximation pair (BAP) problem in two separate directions. One is, for the first time, the ability to consider more than two sets, a task which has been in the mind of researchers for many years without much success in fulfilling it because of a certain theoretical obstacle related to cycles and minimizers of general functionals. The other direction is the ability to handle each set as an intersection of a finite family of sets. We call the resulting problem the ``Best Approximation Tuple (BAT) problem''. The fundamental observation that leads to these extensions is to recognize and handle one set (the ``pivot set'') as different from the remaining sets (the ``satellite sets'') instead of seeking cycles as the minimizers of a target functional. We prove the convergence of the algorithm to the unique solution of the problem in the Euclidean case with strictly convex and compact satellite sets. Because of the lack of Fej\'er monotonicity, our convergence analysis is not standard, and is based on almost unknown properties of orthogonal projections regarding equality and inequality in the definition of nonexpansiveness. 
\end{abstract}


\section{Introduction}\label{sec:Introduction}
\subsection{Background and contributions} In the classical best approximation pair (BAP) problem we are given two nonempty, closed, convex and disjoint subsets $\Omega_1$ and $\Omega_2$ located in a real Hilbert space $X$, and we want to find a pair of points, each from each subset, which realizes the distance between the subsets. In other words, we want to find $a\in \Omega_1$ and $b\in\Omega_2$ such that 
\begin{equation}\label{eq:d(Pmega_1,Omega_2)}
\|a-b\|=\dist(\Omega_1,\Omega_2):=\inf\{\|x-y\|\,|\, x\in \Omega_1, y\in\Omega_2\},
\end{equation}
namely, $(a,b)$ is a minimizer over $\Omega_1\times \Omega_2$ of the function $f:X^2\to [0,\infty)$ defined by $f(x,y):=\|x-y\|$ for all $(x,y)\in X^2$. The BAP problem goes back to the classical 1959 work of Cheney and Goldstein \cite{CheneyGoldstein1959jour} and, briefly, also to the paper \cite{Nicolescu1938jour} of Nicolescu from 1938. Numerous papers discussing this problem in various settings have been published, and the problem has found applications in signal processing, solid modeling, computer graphics, robotics, collision detection, computer aided design, virtual reality and more. For a very partial list of works see, for instance, 
 \cite{AharoniCensorJiang2018jour,BauschkeBorwein1993jour,BauschkeBorwein1994jour,KopeckaReich2004jour,KopeckaReich2012jour,BauschkeCombettesLuke2004jour,
BauschkeSinghWang2022jour,CensorMansourReem2024jour,Dax2006jour,LinGottschalk1998inproc,SatoHirataMaruyamaArita1996inproc,
SonYoonKimElber2020jour,JohnsonCohen1998inproc,PatogluGillespie2002inproc,Quinlan1994inproc,FanWangTongLiTang2024inproc,ChangChoiKimWang2011jour,
GilbertJohnsonKeerthi1988jour,ZeghloulRambeaud1996jour,CameronCulley1986inproc,EhmannLin2000inproc,LinManochaKim2018inbook,
Schwartz1981jour,DobkinKirkpatrick1985jour,Luo2014jour,Narang1991jour,Pai1974jour,Stiles1965b-jour,Xu1988jour,VoiseiZalinescu2011jour,CaseiroFacasVicenteVitoria2019jour}. For a review about the BAP problem (with over 120 citations) in various settings, as well as for many relevant results and applications, see our recent work \cite{ReemCensor2026prep-BAP}.

It is not immediately clear how to generalize the BAP problem to more than two sets, and this task has been in the mind of researchers for many years without much success in fulfilling it, and one can find statements about the ``interesting challenges'' that poses the task of extending results related to the BAP problem (and, implicitly, the BAP problem itself) to more than two nonempty, closed and convex sets \cite[p. 401]{BaillonCombettesCominetti2012jour}. In fact, it has turned out there is a theoretical obstacle for a natural generalization. Indeed, we first recall the known fact, which is implicit in \cite{CheneyGoldstein1959jour} (and follows from \cite[Theorem 2]{CheneyGoldstein1959jour} and additional simple arguments; see also Lemma \bref{lem:FixedPointMin} below for a generalization), that $(a,b)\in \Omega_1\times\Omega_2$ is a BAP relative to $(\Omega_1,\Omega_2)$ if and only if $(a,b)$ is a 2-cycle, namely $a=P_1(b)$ and $b=P_2(a)$. Thus, given $m\in\N$, a natural candidate for being a generalization of the concept of a BAP to any $m+1\geq 2$ sets $\Omega_1,\ldots,\Omega_{m+1}$  is an $(m+1)$-cycle relative to $(\Omega_1,\Omega_2,\ldots,\Omega_{m+1})$, namely a tuple $(a^1,a^2,\ldots,a^{m+1})\in \Omega_1\times\Omega_2\times\cdots\times\Omega_{m+1}$ which satisfies $a^i=P_{i}(a^{i+1})$ for all $i\in \{1,2,\dots,m\}$ and $a^{m+1}=P_{m+1}(a^1)$. 

While there are various relevant results related to cycles, such as  \cite[Section 5]{BauschkeBorweinLewis1997inproc}, \cite[Theorem 2]{GubinPolyakRaik1967jour}, \cite[Theorem 1.6, Corollary 1.7]{Stiles1965b-jour}, \cite[Corollary 4.10, Corollary 4.11]{BehlingBello-CruzSantos2021jour} and \cite[Theorem 9]{AlwadaniBauschkeRevalskiWang2021jour}, none of them shows that when $m+1\geq 3$, then there is a certain functional such that the minimizers of this functional over $\Omega_1\times\Omega_2\times\ldots\times \Omega_{m+1}$ (for all nonempty, closed and convex sets $\Omega_i\subseteq X$, $i\in\{1,2,\ldots,m+1\}$) must be $(m+1)$-cycles relative to $(\Omega_1,\Omega_2,\ldots,\Omega_{m+1})$ and vice versa, as in the case of two sets mentioned above (simple candidates to such a functional, such as $(x^i)_{i=1}^{m+1}\mapsto \|x^1-x^{2}\|+\|x^2-x^3\|+\ldots+\|x^{m}-x^{m+1}\|+\|x^{m+1}-x^1\|$, have been considered and rejected: see \cite[Remarks 5.1.5]{BauschkeBorweinLewis1997inproc} and \cite[p. 36]{Kosmol1987incol} where $m+1=3$). 

Apparently this is not a coincidence, since a remarkable result of Baillon, Combettes and Cominetti \cite[Theorem 2.3]{BaillonCombettesCominetti2012jour} shows that when $X$ has dimension 2 or above (possibly infinite dimension) and $m+1\geq 3$, then there exists no functional $\Phi:X^{m+1}\to\R$ having the property that for all tuple $(\Omega_1,\Omega_2,\ldots,\Omega_{m+1})$ of nonempty, closed and convex sets $\Omega_i\subseteq X$, $i\in\{1,2,\ldots,m+1\}$, the set of $(m+1)$-cycles relative to $(\Omega_1,\Omega_2,\ldots,\Omega_{m+1})$ is the set of minimizers of $\Phi$ over $\Omega_1\times\Omega_2\times\cdots\times \Omega_{m+1}$. This theoretical obstacle has led to statements such as ``the Cheney and Goldstein...result of minimizing the distance between two disjoint sets cannot be extended to more than two sets'' \cite[p. 579, after Theorem 2.18]{CensorZaknoon2018jour} and ``The problem of best approximation pair relative to two sets cannot be extended to more than two sets in view of Baillon, Combettes and Cominetti’s result''  \cite[p. 581, Subsection 2.16]{CensorZaknoon2018jour}. 

We present here a way to bypass this theoretical obstacle by singling out and focusing our attention on one specific component of the tuple of sets which we call the ``pivot set''. Then we pose the problem of minimizing a natural functional, which is defined on $X$ and not on $X^{m+1}$, over the pivot set. 

More precisely, assume for the sake of simplicity that we single out and focus on the $(m+1)$-th index for the pivot set. If $ (\beta_i)_{i=1}^m$ are positive numbers satisfying $\sum_{i=1}^m\beta_i=1$, and if $P_i$ denotes the orthogonal projection onto any remaining sets $\Omega_i$, $i\in \{1,2,\ldots,m\}$ (which we call the ``satellite sets''), then we consider the following minimization problem, wherein $d(x,\Omega_i):=\inf\{\|x-y\|\,|\, y\in \Omega_i\}$ is the distance between the point $x$ and the set $\Omega_i$:

\begin{equation}\label{eq:MinOmega_h}
\min_{x\in\Omega_{m+1}}\;\frac{1}{2}\sum_{i=1}^m\beta_{i}d(x,\Omega_{i})^2=\min_{x\in\Omega_{m+1}}\frac{1}{2}\sum_{i=1}^m\beta_{i}\|x-P_i(x)\|^2.
\end{equation}
Once we have a solution $x^*\in\Omega_{m+1}$ to \beqref{eq:MinOmega_h} at our hands, we consider the tuple $(P_1(x^*),P_2(x^*),\ldots,P_{m}(x^*),x^*)$ as the Best Approximation Tuple (BAT). This tuple is not necessarily an $(m+1)$-cycle when $m+1>2$, but the case $m+1=2$ reduces to the BAP problem considered by Cheney and Goldstein  \cite{CheneyGoldstein1959jour}. 

Indeed, in this case $\beta_1$ must be 1, and $x^*\in\Omega_2$ solves \beqref{eq:MinOmega_h} if and only if $x^*$ solves the minimization problem $\min_{x\in\Omega_{2}}\;\frac{1}{2}d(x,\Omega_{1})^2$, i.e., if and only if $x^*$ is a minimizer of the minimization problem $\min_{x\in\Omega_{2}}d(x,\Omega_{1})$. Since  \cite[Theorem 2]{CheneyGoldstein1959jour} ensures that $x^*$ solves the minimization problem $\min_{x\in\Omega_{2}}d(x,\Omega_{1})$ if and only if $x^*=P_2(P_1(x^*))$, by recalling what we explained a bit after \beqref{eq:d(Pmega_1,Omega_2)}, namely that $(a,b)\in \Omega_1\times\Omega_2$ is a BAP relative to $(\Omega_1,\Omega_2)$ if and only if $(a,b)$ is a 2-cycle, and by letting  $a:=P_1(x^*)$ and $b:=x^*$, we conclude that $x^*\in\Omega_2$ solves \beqref{eq:MinOmega_h} if and only if $(P_1(x^*),x^*)$ is a BAP and a BAT relative to $(\Omega_1,\Omega_2)$. 

Proposition \bref{prop:BAT_is_MultivariateBAT} below shows that a tuple $(y^1,y^2, \ldots,y^{m+1})\in \Omega_1\times\Omega_2\times\cdots\times \Omega_{m+1}$ is a BAT if and only it solves the multivariate  minimization problem 
\begin{equation}\label{eq:MultivariateBAT}
\min_{(x^1,x^2,\ldots,x^{m+1})\in \Omega_1\times\Omega_2\times\cdots\times \Omega_{m+1}}\frac{1}{2}\sum_{i=1}^m\beta_i \|x^{m+1}-x^i\|^2.
\end{equation}
Thus, in contrast with the case forbidden in the Baillon, Combettes and Cominetti result \cite[Theorem 2.3]{BaillonCombettesCominetti2012jour}, there is a (non-symmetric) multivariate functional over $X^{m+1}$, that is $(x^1,x^2,\ldots,x^{m+1})\mapsto 0.5\sum_{i=1}^m\beta_i \|x^{m+1}-x^i\|^2$, $(x^1,x^2,\ldots,x^{m+1})\in X^{m+1}$, whose set of minimizers over $\Omega_1\times\Omega_2\times\cdots\times \Omega_{m+1}$ (for every tuple $(\Omega_1,\Omega_2,\ldots,\Omega_{m+1})$ of nonempty, closed and convex subsets $\Omega_i\subseteq X$, $i\in\{1,2,\ldots,m+1\}$) coincides with the set of ``generalized best approximation pairs'' (namely BATs). Again, a BAT is not necessarily a cycle, and so there is no contradiction with \cite[Theorem 2.3]{BaillonCombettesCominetti2012jour}.

Our generalization \beqref{eq:MinOmega_h} of the BAP problem was inspired by the, so-called, ``inconsistent convex feasibility problem with soft and hard constraints'' and by the generalized Heron problem. Before presenting the former problem, recall first the classical convex feasibility problem (CFP): we are given a family $(\Omega_i)_{i\in I}$ of nonempty, closed and convex subsets in a real Hilbert space $X$, (where $I\neq \emptyset$ and it is finite or infinite) and our goal is to find a feasible point, that is, a point which belongs to the intersection $\cap_{i\in I}\Omega_i$. See, for example,  \cite{BauschkeBorwein1996jour,ButnariuCensorGurfilHadar2008jour,ByrneCensor2001jour,Cegielski2012book,CensorCegielski2015jour,CensorReem2015jour,CensorZenios1997book,Combettes1996jour(CFP),GubinPolyakRaik1967jour}
for related results, algorithms and applications. 

There are, however, various real-world and theoretical instances in which either it is not clear in advance that $\cap_{i\in I}\Omega_i\neq \emptyset$, or it is known that $\cap_{i\in I}\Omega_i=\emptyset$ and yet one wants to find a point which somehow generalizes the notion of a feasible point. A typical example for this is the inconsistent convex feasibility problem with soft and hard constraints \cite{CombettesBondon1999jour}. In this problem the sets $\Omega_{i},i\in I$ are divided into two groups: one group is the hard constraints group $\Omega_i$, $i\in I_{\textnormal{hard}}\neq\emptyset$, and the second group is the soft constraints group $\Omega_i$, $i\in I_{\textnormal{soft}}\neq\emptyset$, where $I_{\textnormal{hard}}\cup I_{\textnormal{soft}}=I$, $I_{\textnormal{hard}}\cap I_{\textnormal{soft}}=\emptyset$, and $I$ has at least two indices.  The goal is to find a point $x^* \in \cap_{i\in I_{\textnormal{hard}}}\Omega_i$ (assumed or known to be nonempty) which violates as little as possible the constraints $\Omega_i$, $i\in I_{\textnormal{soft}}$. More precisely, we want to find  a point $x^* \in \cap_{i\in I_{\textnormal{hard}}}\Omega_i$ which is a minimizer, over $\cap_{i\in I_{\textnormal{hard}}}\Omega_i$, of the functional $\Phi_{\textnormal{soft}}: X\to \R$ defined by 
\begin{equation}\label{eq:Phi_soft}
\Phi_{\textnormal{soft}}(x):=\frac{1}{2}\sum_{i\in I_{\textnormal{soft}}}\beta_id(x,\Omega_i)^2=\frac{1}{2}\sum_{i\in I_{\textnormal{soft}}}\beta_i\|x-P_i(x)\|^2, \quad \forall x\in X, 
\end{equation}
where, for each  ${i\in I_{\textnormal{soft}}}$, $P_i:X\to X$ is the orthogonal projection onto $\Omega_i$, $\beta_i>0$ and $\sum_{i\in I_{\textnormal{soft}}}\beta_i=1$. 

The functional $\Phi_{\textnormal{soft}}$ is a proximity function (see, e.g., \cite[p. 28]{Cegielski2012book}). It has the property that $\Phi_{\textnormal{soft}}(x)=0$ if and only if $x\in \cap_{i\in I_{\textnormal{soft}}}\Omega_i$, and $\Phi_{\textnormal{soft}}(x)>0$ if and only if $x\notin \Omega_i$ for at least one index $i\in I_{\textnormal{soft}}$. One can think about $\Phi_{\textnormal{soft}}$ as a cost function which assigns to each $x\in X$ the price tag $\Phi_{\textnormal{soft}}(x)$ which represents by how much $x$ violates the constraints: the higher the violation, the higher the price, and we want to find a point $x^* \in \cap_{i\in I_{\textnormal{hard}}}\Omega_i$ having the least price tag. Alternatively, $\Phi_{\textnormal{soft}}$ can be thought of as being an energy function which assigns to each $x\in X$ the energy level $\Phi_{\textnormal{soft}}(x)$, and our goal is to find a point $x^* \in \cap_{i\in I_{\textnormal{hard}}}\Omega_i$ having the lowest possible energy level. 

The inconsistent convex feasibility problem with soft and hard constraints has applications in signal processing \cite{CombettesBondon1999jour,GoldburgMarks1985jour,YoulaVelasco1986jour}, network bandwidth allocation problem \cite{Iiduka2012jour-a}, power control \cite{Iiduka2012jour-b}, network analysis \cite{Iiduka2020jour}, sound velocity reconstruction \cite{HuangLi2004jour} and more; see also \cite{AttouchPeypouquetRedont2014jour,Cegielski2012book,CensorElfvingKopfBortfeld2005jour,CensorBortfeldMartinTrofimov2006jour,
CensorZaknoonZaslavski2021jour,Combettes1994jour,Combettes2003jour,Combettes2004jour,CombettesWajs2005jour,CombettesPesquet2008jour,Combettes2013jour,CombettesGlaudin2019jour,CombettesWoodstock2022jour} for various algorithms and other results related to this problem, and see \cite{CensorZaknoon2018jour} for a general review related to the  inconsistent convex feasibility problem.

Back to \beqref{eq:MinOmega_h}, one can see that it is a particular case of inconsistent convex feasibility problem with soft and hard constraints where $I_{\textnormal{hard}}=\{m+1\}$ and $I_{\textnormal{soft}}=\{1,2,\ldots,m\}$. Of course, we could let  $I_{\textnormal{hard}}:=\{h\}$ and  $I_{\textnormal{soft}}:=I\backslash\{h\}$ for any $h\in \{1,2,\ldots,m+1\}$, and by doing this, we would have obtained a different functional to be minimized and a different generalization of the BAP problem. 

The second problem which has inspired us is the generalized Heron problem \cite{MordukhovichNamSalinas2012jour-a,MordukhovichNamSalinas2012jour-b,ChiLange2014jour} (see also \cite{YoulaVelasco1986jour}). Its original version (the classical Heron problem) goes back to antiquity, and it has applications in optics and engineering. Here instead of minimizing the quadratic-type functional \beqref{eq:MinOmega_h} one minimizes, over the pivot set, a non-weighted sum of the distances between a point and the satellite sets: 
\begin{equation}
\min_{x\in\Omega_{m+1}}\;\sum_{i=1}^{m}d(x,\Omega_{i}).\label{eq:GenHeron}
\end{equation}
In addition to extending the BAP problem to any finite number of sets (greater than or equal to 2), we also extend the problem in another aspect, by considering sets $\Omega_i$ having the property that each of them is an intersection of finitely many sets, namely $\Omega_i=\cap_{j=1}^{m_i}\Omega_{i,j}$ for some sets $\Omega_{i,j}$, where $i\in \{1,2,\ldots,m\}$, $m_i\in\N$ and $j\in \{1,2,\ldots,m_i\}$. Such a scenario appears, for instance, in bandwidth allocation problems \cite[Subsections 2.1 and 2.2]{Iiduka2012jour-a} and in reconstruction of sound velocity distribution \cite[Sections III, IV]{HuangLi2004jour}. See also \cite[Fig. 1]{CombettesBondon1999jour} and \cite[Subsection 6.1]{Combettes1996jour(CFP)}. 

Our extension, mentioned above, is what we call the Best Approximation Tuple (BAT) problem. Inspired by the ``Alternating Simultaneous Halpern--Lions--Wittmann--Bauschke (A-S-HLWB)'' algorithm that we presented in \cite{CensorMansourReem2024jour}, and by other material which appears there, we formulate precisely the BAT problem (Problem \bref{prob:The-Multi-Sets-Generalized-1} below), formulate conditions which ensure the existence and uniqueness of a solution to this problem, present an algorithmic scheme for solving the problem (Algorithm \bref{alg:MSGBAP} below), and prove its convergence under some reasonable assumptions (Theorem \bref{thm:Main} below; see also Remark \bref{rem:Better_Than_whP} for an extension), among them that the space is Euclidean and the satellite sets are strictly convex and compact. 

Algorithm \bref{alg:MSGBAP} below has the advantage that instead of projecting onto the individual sets $\Omega_i$, $i\in\{1,2,\ldots,m\}$, a task which may be demanding from the computational point of view, one projects onto the building bricks $\Omega_{i,j}$, $i\in \{1,2,\ldots,m\}$, $j\in \{1,2,\ldots,m_i\}$ which induce the intersection set $\Omega_i$, $i\in\{1,2,\ldots,m\}$, a task which is less expensive, particularly when these building bricks have a simple form. 

Because of the lack of Fej\'er monotonicity, our convergence analysis is not standard, and is based on almost unknown properties of orthogonal projections regarding equality and inequality in the definition of nonexpansiveness: see Lemma \bref{lem:P(x)-P(y)=x-y}, Remark \bref{rem:StrictlyNonexpansive}, Lemma \bref{lem:Rectangle} and Lemma \bref{lem:|P_S(x)-P_S(y)|<|x-y|} below. These properties allow us to derive strict estimates at crucial stages (such as the inequality  \beqref{eq:r_epsilon} and the inequality $r_{\varepsilon}+\|P_i(x^k)-P_i(x^*)\|<\varepsilon$ for all $i\in\{1,2\ldots,m\}$ mentioned below  \beqref{eq:x^(k+1)-x^*}). 

\subsection{Paper layout:} In Section \bref{sec:Preliminaries} we present necessary terminology and notation needed for the rest of the paper. In Section \bref{sec:The-Multi-Sets-Generalized} we present precisely the BAT problem, the algorithm for solving it, the main existence, uniqueness and convergence theorem (Theorem \bref{thm:Main}), a characterization property of BATs as minimizers of a multivariate target functional (Proposition \bref{prop:BAT_is_MultivariateBAT}), and we also compare briefly our algorithm and results to the ones presented by Cheney and Goldstein \cite{CheneyGoldstein1959jour}. In Section \bref{sec:Auxiliary} we formulate and prove various auxiliary assertions, and in Section \bref{sec:Convergence-analysis} we present the proof of Theorem \bref{thm:Main}. 

\section{Preliminaries}\label{sec:Preliminaries}
Our setting is a real Hilbert space $(X,\|\cdot\|)$ with an inner product $\langle \cdot,\cdot \rangle$ and $X\neq\{0\}$, although some of the definitions below hold in a normed space or a metric space setting, and from time to time we will also consider these and other (perhaps more general) settings. 

Given a nonempty subset $S\subseteq X$ and an operator $F:S\to X$, we say that $F$ is nonexpansive on $S$ if $\|F(x)-F(y)\|\leq \|x-y\|$ for all $(x,y)\in S^2$. We say that $F$ is contractive on $S$ in the sense of Edelstein, or Edelstein-contractive, or simply contractive (following Edelstein \cite[p. 74]{Edelstein1962jour}), if $\|F(x)-F(y)\|<\|x-y\|$ whenever $x\neq y$ and $(x,y)\in S^2$. Note that in \cite[p. 69, Definition 4.1(iii)]{BauschkeCombettes2017book} a different term is used for describing Edelstein-contractive operators: ``strictly nonexpansive'' (we chose to not use this term because of a reason which is explained in Remark 
\bref{rem:StrictlyNonexpansive} below). An Edelstein-contractive mapping on $S$ is nonexpansive on $S$ since given $(x,y)\in S^2$, obviously $\|F(x)-F(y)\|\leq\|x-y\|$ whenever $x=y$ or $x\neq y$.

Given $\nu>0$, we say that $F$ is $\nu$-firmly nonexpansive on $S$ if $\|F(x)-F(y)\|^2\leq \|x-y\|^2-\nu\|(x-y)-(F(x)-F(y))\|^2$ for all $(x,y)\in S^2$. If  $F$ is $1$-firmly nonexpansive on $S$, then we also say that $F$ is firmly nonexpansive. Various examples of $\nu$-firmly nonexpansive operators for $\nu>0$  can be found in \cite[Chapter 2, Section 2.2]{Cegielski2012book}.
 
If $S$ is nonempty, closed and convex, then it is well-known \cite[Theorem 1.2.3 on p. 18]{Cegielski2012book} that for all $x\in X$ there exists a unique point in $S$, denoted by $P_S(x)$, such that $\|x-P_S(x)\|=d(x,S):=\inf\{\|x-y\|\,|\,y\in S\}$. We call the operator $P_S:X\to S$, which assigns to each $x\in X$ the unique point $P_S(x)$, the orthogonal projection onto $S$ (also called best approximation projection, metric projection, projector, and proximity map). It is well known that $P_S$ is firmly nonexpansive \cite[Theorem 2.2.21(iii) on p. 76 and Theorem 2.2.10(v) on p. 70]{Cegielski2012book} and satisfies the so-called Kolmogorov property $\langle x- P_S(x), y-P_S(x)\rangle\leq 0$ for all $x\in X$ and $y\in S$ (see, e.g., \cite[Theorem 2.2.21(ii) on p. 76 and Theorem 2.2.5 on p. 67]{Cegielski2012book} or \cite[The Lemma on p. 448]{CheneyGoldstein1959jour}). 
 
We say that $S$ is strictly convex if for all distinct points $a\in S$ and $b\in S$, the open interval $(a,b)$ connecting $a$ and $b$ is contained in the interior $\Int(S)$ of $S$. We denote by $\Id$ the identity operator on $X$ which satisfies $\Id(x)=x$ for each $x\in X$. Given $x\in X$ and $r>0$, we denote by $B[x,r]:=\{y\in X\,|\,\|y-x\|\leq r\}$ and $B(x,r):=\{y\in X\,|\, \|y-x\|<r\}$ the closed and open balls, respectively, of radius $r$, about $x$ (and $B[x,0]:=\{x\}$).

Given $m\in\N$ and a vector $\beta:=(\beta_i)_{i=1}^m\in\R^m$, we say that $\beta$ is a ``positive weight vector'' if $\beta_{i}>0$ for all $i\in\{1,2,\ldots,m\}$ and $\sum_{i=1}^{m}\beta_{i}=1$. A sequence $(\tau_{k})_{k=0}^{\infty}$ of real numbers is said to be a ``steering parameter sequence'' whenever the following conditions hold: $\tau_k\in (0,1)$ for all $k\in \N\cup\{0\}$,  $\lim_{k\to\infty}\tau_{k}=0$, $\sum_{k=0}^{\infty}\tau_{k}=\infty$, and $\sum_{k=0}^{\infty}\left|\tau_{k+1}-\tau_{k}\right|<+\infty$. For example, the sequence  $(\tau_{k})_{k=0}^{\infty}$ which satisfies $\tau_{k}:=1/(k+1)$ for all $k\in\mathbb{N}\cup\{0\}$ is a steering parameter sequence. 

\section{The best approximation tuple problem}\label{sec:The-Multi-Sets-Generalized}
In this section we formulate the BAT problem (Subsection \bref{subsec:Problem-definition}), our algorithm for solving it (Subsection \bref{subsec:The-Multi-Sets-Generalized}), and the main existence, uniqueness and convergence theorem as well as a characterization of a BAT (Subsection \bref{subsec:ConvergenceTheorem}). In Subsection \bref{subsec:Comparison} we make a short comparison between our algorithm and results to the ones obtained by Cheney and Goldstein \cite{CheneyGoldstein1959jour}.

\subsection{Problem formulation}\label{subsec:Problem-definition} (See Figure \bref{fig:MSGBAP-Illustration} for an illustration.)

\begin{problem}\label{prob:The-Multi-Sets-Generalized-1}
\textbf{The Best Approximation Tuple (BAT) problem:}  In a real Hilbert space $X$ we are given $m\in \N$ sets $\left(\Omega_{i}\right)_{i=1}^{m}$ (which we call ``the satellite sets''), each of which is itself a nonempty intersection of a given family $\left(\Omega_{i,j}\right)_{j=1}^{m_{i}}$ of $m_{i}\in \N$ closed and convex subsets of $X$, i.e., $\Omega_{i}:=\cap_{j=1}^{m_{i}}\Omega_{i,j}\neq\emptyset$ for all $i\in\{1,2,\ldots,m\}$. We are also given another set (which we call ``the pivot set)'' $\Omega_{m+1}$ which is nonempty, closed and convex, and a positive weight vector $(\beta_i)_{i=1}^m$. The goal of the ``Best Approximation Tuple'' (BAT) problem is to solve the following constrained minimization problem: 
\begin{equation}\label{eq:MinProb}
\min_{x\in\Omega_{m+1}}\;\frac{1}{2}\sum_{i=1}^{m}\beta_{i}(d(x,\Omega_{i}))^2.
\end{equation}
In other words, if $\Phi: X\to \R$ is defined by 
\begin{equation}\label{eq:Phi}
\Phi(x):=\frac{1}{2}\sum_{i=1}^m\beta_id(x,\Omega_i)^2=\frac{1}{2}\sum_{i=1}^m\beta_i\|x-P_i(x)\|^2, \quad \forall x\in X, 
\end{equation}
where $P_i:X\to X$ is the orthogonal projection onto $\Omega_i$, $i\in\{1,2,\ldots,m\}$, then the problem is to provide conditions which ensure the existence and perhaps uniqueness of a minimizer of $\Phi$ over $\Omega_{m+1}$, and to present algorithms for finding such a minimizer. 
\end{problem}

\begin{figure}[htb]
\begin{minipage}{1\textwidth}
\begin{center}{\includegraphics[trim=170 90 60 80, clip=true, scale=0.55]{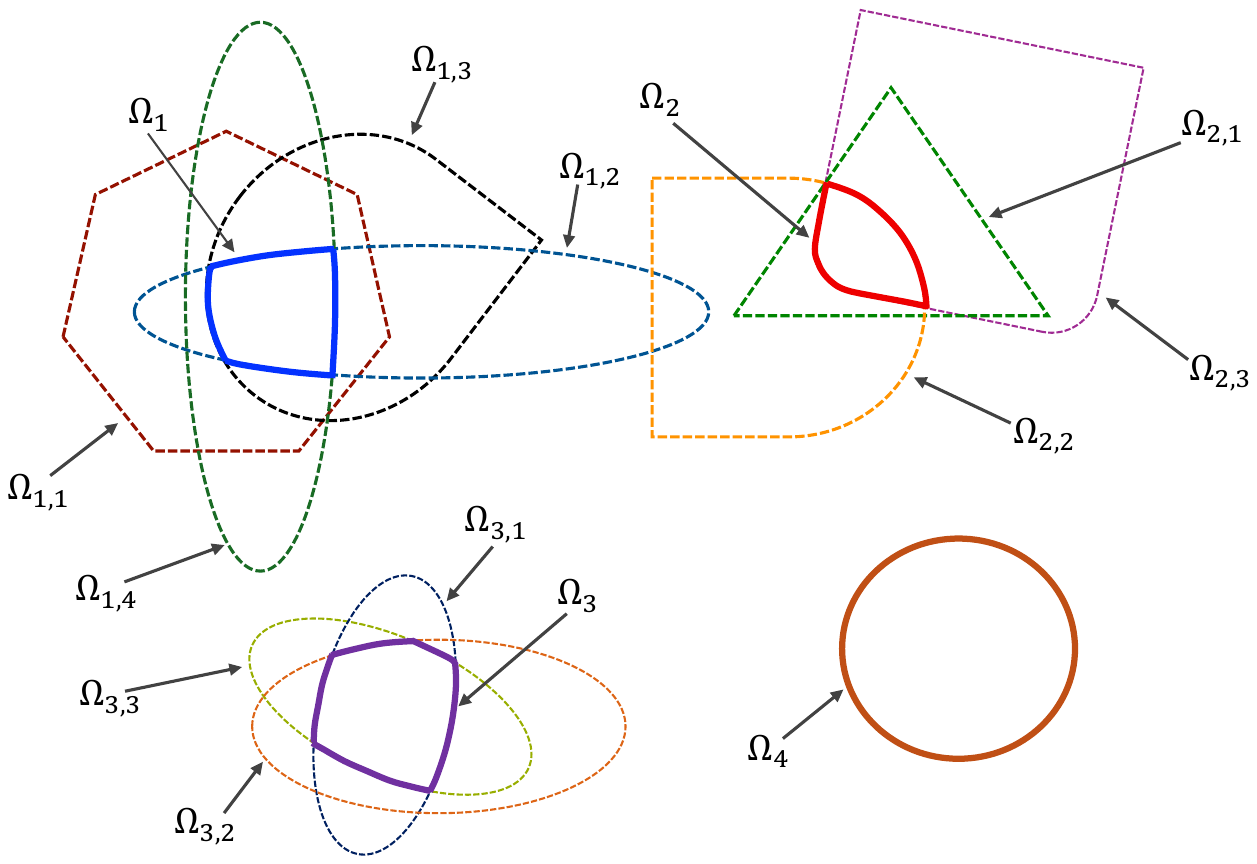}}
\end{center}
 \caption{Illustration of Problem \bref{prob:The-Multi-Sets-Generalized-1} in the Euclidean plane. Here $m=3$, $m_1=4$, $m_2=3$, $m_3=3$, the boundaries of the sets $\Omega_1,\Omega_2,\Omega_3,\Omega_4$ are represented by solid lines, and the boundaries of the sets $\Omega_{i,j}$ are represented by dashed lines, $i\in \{1,2,\ldots,m\}$, $j\in \{1,2,\ldots,m_i\}$. The satellite sets are $\Omega_1,\Omega_2,\Omega_3$ and the pivot set is $\Omega_4$. }
\label{fig:MSGBAP-Illustration}
\end{minipage}
\end{figure}

\begin{defin}\label{def:BAT}
A tuple $(y^1,y^2,\ldots,y^{m+1})\in \Omega_1\times\Omega_2\times\ldots\times \Omega_{m+1}$ is called a BAT (relative to $(\Omega_i)_{i=1}^{m+1}$) if $y^{m+1}$ solves \beqref{eq:MinOmega_h} and $y^i=P_i(y^{m+1})$ for each $i\in\{1,2,\ldots,m\}$. 
\end{defin}

\subsection{The Multiple-Sets Best Approximation Tuple algorithm\label{subsec:The-Multi-Sets-Generalized}}

Here we propose a projection-type algorithm for the solution of the BAT problem. 

\begin{alg}\label{alg:MSGBAP} {\bf (Algorithm for the BAT problem)}\\
\begin{enumerate}
\item \noindent\textbf{Input: } A real Hilbert space $(X,\|\cdot\|)$, a natural number $m$, a set of $m$ natural numbers $m_i$ and sequences $(\tau_{i,t})_{t=0}^{\infty}$ of steering parameters for each $i\in\{1,2,\ldots,m\}$, positive weight vectors $w_i=(w_{i,j})_{j=1}^{m_i}$ where $i\in\{1,2,\ldots,m\}$, closed and convex subsets $\Omega_i$ (the satellite sets) and $\Omega_{i,j}$ such that $\Omega_i=\cap_{j=1}^{m_i}\Omega_{i,j}\neq\emptyset$ for all $i\in\{1,2,\ldots,m\}$ and $j\in \{1,2,\ldots,m_j\}$, a nonempty closed and convex subset $\Omega_{m+1}$ (the pivot set), a positive weight vector $\beta=(\beta_i)_{i=1}^{m}$, and two sequences $(\lambda_k)_{k=0}^{\infty}$ and $(\gamma_k)_{k=0}^{\infty}$ of real numbers in $(0,1]$.\vspace*{0.1cm}   

\item \noindent\textbf{Initialization}: Choose an arbitrary starting point \textit{$x^{0}\in X$}.\vspace*{0.1cm}

\item \noindent\textbf{Iterative Step}: For each $k\in\N\cup\{0\}$ and $i\in\{1,2,\ldots,m\}$, define
\begin{equation}\label{eq:ApproxProjection}
\wh{P}_{i,k}(x):=\prod_{t=0}^{k}\left(\tau_{i,t}x+(1-\tau_{i,t})\sum_{j=1}^{m_i}w_{i,j}P_{i,j}(x)\right), \quad \forall x\in X,
\end{equation}
\begin{equation}\label{eq:G_k}
G_k(x):=(1-\lambda_{k})x+\lambda_{k}P_{m+1}\left((1-\gamma_k)x+\gamma_k\sum_{i=1}^{m}\beta_i\wh{P}_{i,k}(x)\right),\quad \forall x\in X,
\end{equation}
where $P_{i,j}$ is the orthogonal projection onto the closed and convex subset $\Omega_{i,j}$, $i\in\{1,2,\ldots,m\}$ and $j\in \{1,2,\ldots,m_j\}$ and $P_{m+1}$ is the orthogonal projection onto the pivot set $\Omega_{m+1}$. Now, given a current iterate $x^k$, $k\in\N\cup\{0\}$, define the next iterate $x^{k+1}$ by the following iterative process: 
\begin{equation}\label{eq:GBAP}
x^{k+1}:=G_k(x^k)=(1-\lambda_{k})x^k+\lambda_{k}P_{m+1}\left((1-\gamma_k)x^k+\gamma_k\sum_{i=1}^{m}\beta_i\wh{P}_{i,k}(x^k)\right).
\end{equation}
\end{enumerate}
\end{alg}

\begin{remark}
\begin{enumerate}[(i)]
\item The vector $(\beta_i)_{i=1}^m$ in the functional $\Phi$ from \beqref{eq:Phi} is assumed to be a positive weight vector. This is not a real restriction as long as the coefficients $\beta_i$, $i\in \{1,2,\ldots,m\}$ are assumed to be positive. Indeed, suppose that instead of a positive weight vector we just assume that $(\beta_i)_{i=1}^m$ is a vector of positive numbers. Then by letting $\overline{\beta}_i:=\beta_i/(\sum_{j=1}^m\beta_j)$ for each $i\in\{1,2,\ldots,m\}$ and $\overline{\Phi}(x):=\frac{1}{2}\sum_{i=1}^m \overline{\beta}_i d(x,\Omega_i)^2$, $x\in X$, we see that $(\overline{\beta}_i)_{i=1}^m$ is a weight vector, and an immediate verification shows that $\overline{\Phi}$ and $\Phi(x):=\frac{1}{2}\sum_{i=1}^m \beta_i d(x,\Omega_i)^2$, $x\in X$, have the same minimizers over any subset $S$ of $X$. 
\item The iterative scheme \beqref{eq:GBAP} is inspired by a different iterative scheme presented in \cite[(20) on p. 2463]{CombettesBondon1999jour}, whose goal is to solve the inconsistent convex feasibility problem with soft and hard constraints, where each subset is formed by itself (namely $m_i=1$ for all $i\in\{1,2,\ldots,m\}$).  
\end{enumerate}
\end{remark}

\subsection{The main existence, uniqueness and convergence theorem}\label{subsec:ConvergenceTheorem}
The proof of the next theorem, which is the main convergence result, is long and is based on several lematta which we formulate and prove in the next sections.
\begin{Thm}\label{thm:Main} In the setting of Algorithm \bref{alg:MSGBAP}, suppose that either $\Omega_{m+1}$ is bounded or that $\Omega_{m+1}$ is unbounded and $\Omega_q$ is bounded for some $q\in \{1,2,\ldots,m\}$. Then \beqref{eq:MinProb} has at least one solution. If, in addition, $\Omega_s$ is strictly convex for some $s\in \{1,2,\ldots,m\}$ and also $\Omega_s\cap \Omega_{m+1}=\emptyset$, then \beqref{eq:MinProb} has a unique solution $x^*$, and $x^*$ has the property that it is also the unique fixed point of each of the functions $G_{\lambda,\gamma}:X\to X$, $\lambda\in (0,1]$, $\gamma\in (0,2)$, defined by 
\begin{equation}\label{eq:G_lambda_gamma}
G_{\lambda,\gamma}(x):=(1-\lambda)x+\lambda P_{m+1}((1-\gamma)x+\gamma\sum_{i=1}^m \beta_i P_i(x)),\quad \forall x\in X.
\end{equation}
Assume further that $X$ is finite dimensional, that there is some $\rho>0$ such that $\Omega_{m+1}\cup(\cup_{i=1}^m\cup_{j=1}^{m_i}\Omega_{i,j})\subseteq B[0,\rho]$, that $\Omega_i$ is strictly convex and satisfies $\Omega_i\cap \Omega_{m+1}=\emptyset$ for each $i\in\{1,2,\ldots,m\}$,    and that  $\lambda_{\infty}:=\lim_{k\to\infty}\lambda_k$ and $\gamma_{\infty}:=\lim_{k\to\infty}\gamma_k$ exist and satisfy $\lambda_{\infty}>0$ and $\gamma_{\infty}>0$. If $x^0\in \Omega_{m+1}$, then any sequence $(x^k)_{k=0}^{\infty}$ generated by Algorithm  \bref{alg:MSGBAP} is contained in $\Omega_{m+1}$ and converges to $x^*$. 
\end{Thm}

\begin{remark}
\begin{enumerate}[(i)]
\item Non-uniqueness of the optimal solution $x^*$ of \beqref{eq:MinProb} can occur without imposing the strict convexity assumption on the sets in Theorem \bref{thm:Main}. Indeed, suppose that $X$ is the Euclidean plane, $\Omega_1:=[-10,-\eta_1]\times [0,1]$, $\Omega_2:=[\eta_2,10]\times [0,1]$, $\eta_1,\eta_2 \in (0, 10)$, $\Omega_3:=[-0.5\eta_1,0.5\eta_2]\times [0,1]$, $m=2$, $m_1=m_2=1$, and $(\beta_1,\beta_2)$ is an arbitrary positive weight vector. In this case there exists a solution $x^*=(x^*_1,x^*_2)$ to  \beqref{eq:MinProb} because of considerations mentioned in Lemma \bref{lem:Existence_Min} below. Since $\Phi$ is constant on intervals of the form $\{t\}\times [0,1]$ for each $t\in [-0.5\eta_1,0.5\eta_2]$ as a simple verification shows (because $d((x_1,x_2),\Omega_1)=x_1+\eta_1$ and $d((x_1,x_2),\Omega_2)=\eta_2-x_1$ for all $(x_1,x_2)\in \Omega_3$), it follows that for every $x_2\in [0,1]$ the pair $(x^*_1,x_2)$ solves \beqref{eq:MinProb}.

\item Non-uniqueness of the optimal solution of \beqref{eq:MinProb} can occur by replacing the strictly convex functional $\Phi$ defined in \beqref{eq:Phi} by the non-strictly convex functional $\Psi(x):=\frac{1}{2}\sum_{i=1}^m \beta_i d(x,\Omega_i)$, $x\in X$. Indeed, just take $X:=\R$, $m:=2$, $m_1:=m_2:=1$, $\Omega_1:=\{0\}$, $\Omega_2:=\{2\}$, $\Omega_3:=[0.5,1.5]$, $\beta_1:=\beta_2:=0.5$. In this case $\Psi(x)=0.5$ for all $x\in \Omega_3$. Similarly, if $X:=\R^2$, $m:=2$, $m_1:=m_2:=1$,  $\Omega_1:=B[(-1,0),1]$, $\Omega_2:=B[(3,0),1]$, $\Omega_3:=B[(1,0),0.5]$, $\beta_1:=\beta_2:=0.5$, then $\Psi(x)=0.5\leq \Psi(z)$ for all $x\in [0.5,1.5]\times\{0\}\subset \Omega_3$ and all $z\in \Omega_3$.

\item While in the last part of Theorem \bref{thm:Main} we require that $\Omega_i$ is strictly convex and satisfies $\Omega_i\cap \Omega_{m+1}=\emptyset$ for each $i\in\{1,2,\ldots,m\}$, we do not require that $\Omega_{i,j}$ is strictly convex and satisfies $\Omega_{i,j}\cap \Omega_{m+1}=\emptyset$ and $\Omega_{i,j}\cap \Omega_{i',j'}=\emptyset$ for each $i,i'\in\{1,2,\ldots,m\}$ and each $j\in\{1,2,\ldots, m_i\}$ and $j'\in \{1,2,\ldots,m_{i'}\}$. Hence, in particular, it may happen that $\Omega_{i_1,j_1}$ intersects both $\Omega_{i_2,j_2}$ and $\Omega_{m+1}$ for some $i_1,i_2\in \{1,2,\ldots,m\}$ and some $j_1\in\{1,2,\ldots,m_1\}$ and $j_2\in \{1,2,\ldots,m_2\}$: see Figure \bref{fig:MSGBAP-Illustration}. We note that a sufficient condition for $\Omega_i$ to be strictly convex is that  $\Omega_{i,j}$ is strictly convex for all $j\in\{1,2,\ldots,m_i\}$, but, as shown in Figure \bref{fig:MSGBAP-Illustration}, this condition is not necessary. 
\item In the special case where $m=1$ in Theorem \bref{thm:Main} we do not need to impose the boundedness of either $\Omega_1$ or $\Omega_2$ for the existence of a solution to \beqref{eq:MinProb} (which, as explained earlier, is equivalent to the existence of a BAP relative to $(\Omega_1,\Omega_2)$): see \cite[Theorem 5.1]{ReemCensor2026prep-BAP} for less restrictive assumptions.  
\end{enumerate}
\end{remark}

The next proposition characterizes a BAT as the solution of a multivariate minimization problem. This characterization can serve as an alternative definition of a BAT, a definition which actually holds in any normed space and with any tuple of nonempty sets $(\Omega_i)_{i=1}^{m+1}$ (not necessarily closed and convex). 
\begin{proposition}\label{prop:BAT_is_MultivariateBAT}
A tuple $(y^1,y^2,
\ldots,y^{m+1})\in \Omega_1\times\Omega_2\times\cdots\times \Omega_{m+1}$ is a BAT if and only it solves the minimization problem \beqref{eq:MultivariateBAT}. 
\end{proposition}

\begin{proof}
Indeed, suppose by way of contradiction that $(y^1,y^2,\ldots,y^{m+1})$ is a BAT which does not solve \beqref{eq:MultivariateBAT}. Then $\sum_{i=1}^m\beta_i \|x^{m+1}-x^i\|^2<\sum_{i=1}^m\beta_i \|y^{m+1}-y^i\|^2$ for some $(x^1,x^2,\ldots,x^{m+1})\in \Omega_1\times\Omega_2\times\cdots\times \Omega_{m+1}$. Since $x^i\in \Omega_i$ for each $i\in\{1,2,\ldots,m\}$, we have $\|x^{m+1}-P_i(x^{m+1})\|\leq \|x^{m+1}-x^i\|$ by the definition of the orthogonal projection; thus $0.5\sum_{i=1}^m\beta_i \|x^{m+1}-P_i(x^{m+1})\|^2\leq 0.5\sum_{i=1}^m\beta_i \|x^{m+1}-x^i\|^2$. Since $(y^1,y^2,\ldots,y^{m+1})$ is a BAT, we have $y^i=P_i(y^{m+1})$ for all $i\in\{1,2,\ldots,m\}$; as a result, we conclude from the previous lines that $0.5\sum_{i=1}^m\beta_i \|x^{m+1}-P_i(x^{m+1})\|^2<0.5\sum_{i=1}^m\beta_i \|y^{m+1}-P_i(y^{m+1})\|^2$, in contrast to the assumption $y^{m+1}$ solves \beqref{eq:MinOmega_h} (since $(y^1,y^2,\ldots,y^{m+1})$ is a BAT). This contradiction implies that any BAT must solve \beqref{eq:MultivariateBAT}. 

Conversely, suppose that $(y^1,y^2,\ldots,y^{m+1})\in \Omega_1\times\Omega_2\times\cdots\times \Omega_{m+1}$ solves \beqref{eq:MultivariateBAT} and assume by way of contradiction that it is not BAT. Then there are two possibilities (see Definition \bref{def:BAT}): either $y^j\neq P_j(y^{m+1})$ for some $j\in \{1,2,\ldots,m\}$ or $y^i=P_i(y^{m+1})$ for all $i\in \{1,2,\ldots,m\}$ but $y^{m+1}$ does not solve \beqref{eq:MinOmega_h}. 

Consider the first possibility. Basic properties of the orthogonal projection and the fact that $y^i\in \Omega_i$ for all $i\in \{1,2,\ldots,m\}$ imply the inequality $\|y^{m+1}-P_i(y^{m+1})\|\leq \|y^{m+1}-y^i\|$ for each $i\in \{1,2,\ldots,m\}$, where for $i=j$ we have a strict inequality. Hence $0.5\sum_{i=1}^m\beta_i\|y^{m+1}-P_i(y^{m+1})\|^2<0.5\sum_{i=1}^m\beta_i\|y^{m+1}-y^i\|^2$, namely the target functional which appears in \beqref{eq:MultivariateBAT} has a strictly lower value at the tuple $(x^i)_{i=1}^{m+1}:=(P_1(y^{m+1}),P_2(y^{m+1}),\ldots,P_m(y^{m+1}),y^{m+1})$, which is in $\Omega_1\times\Omega_2\times\cdots\times \Omega_{m+1}$, than it has at its assumed minimizer $(y^1,y^2,\ldots,y^{m+1})$, a contradiction. Thus any solution of \beqref{eq:MultivariateBAT} must be a BAT in the first possibility.

Consider the second possibility in which $y^i=P_i(y^{m+1})$ for all $i\in \{1,2,\ldots,m\}$ but $y^{m+1}$ does not solve \beqref{eq:MinOmega_h}. This possibility implies that there is some $x^{m+1}\in\Omega_{m+1}$ such that $0.5\sum_{i=1}^m\beta_i \|x^{m+1}-P_i(x^{m+1})\|^2<0.5\sum_{i=1}^m\beta_i \|y^{m+1}-P_i(y^{m+1})\|^2$. Because $(y^1,y^2,\ldots,y^{m+1})\in \Omega_1\times\Omega_2\times\cdots\times \Omega_{m+1}$ solves \beqref{eq:MultivariateBAT}, we obtain the inequality $0.5\sum_{i=1}^m\beta_i \|y^{m+1}-y^i\|^2\leq 0.5\sum_{i=1}^m\beta_i \|x^{m+1}-x^i\|^2$ for every $i\in \{1,2,\ldots,m\}$ and every $x^i\in\Omega_i$, and in particular for $x^i:=P_i(x^{m+1})$. As a result we obtain the inequality $0.5\sum_{i=1}^m\beta_i \|y^{m+1}-P_i(y^{m+1})\|^2\leq 0.5\sum_{i=1}^m\beta_i \|x^{m+1}-P_i(x^{m+1})\|^2$ and this is a contradiction to a previous line from which we already know that $0.5\sum_{i=1}^m\beta_i \|x^{m+1}-P_i(x^{m+1})\|^2<0.5\sum_{i=1}^m\beta_i \|y^{m+1}-P_i(y^{m+1})\|^2$. We conclude that any solution to \beqref{eq:MultivariateBAT} must be a BAT in the second possibility too. 
\end{proof}

\subsection{ A short comparison with Cheney and Goldstein \cite{CheneyGoldstein1959jour}}\label{subsec:Comparison}
It is worth comparing briefly our results and algorithm with the ones obtained by Cheney and Goldstein in their celebrated 1959 paper \cite{CheneyGoldstein1959jour}. 

First, if in the BAT Problem \bref{prob:The-Multi-Sets-Generalized-1} we have $m=1=m_1$, then by what we explained earlier (below \beqref{eq:MinOmega_h}) the BAT problem,  namely Problem \bref{prob:The-Multi-Sets-Generalized-1}, reduces to BAT problem  considered in \cite{CheneyGoldstein1959jour}. Hence the BAT problem generalizes the BAP problem from two subsets to any finite number (greater than or equal 2) of subsets and from stand-alone satellite subsets to satellite subsets that each one of them is obtained from intersections of finitely many subsets.  

Second, if in Algorithm \bref{alg:MSGBAP} we have $m=1=m_1$, assume that $\lambda_k=\gamma_k=1$ for all $k\in\N\cup\{0\}$, and replace $\wh{P}_{1,k}$ by $P_1$ (see also Remark \bref{rem:Better_Than_whP} below for the justification for this latter replacement), then \beqref{eq:GBAP} leads to $x^{k+1}=P_2P_1(x^k)$ for each $k\in\N\cup\{0\}$. In other words, we restore the alternating projection method presented in \cite{CheneyGoldstein1959jour}.

Third, if we take $\lambda:=1=:\gamma$ in  Lemma \bref{lem:FixedPointMin} below, then we see that Lemma \bref{lem:FixedPointMin} generalizes \cite[Theorem 2]{CheneyGoldstein1959jour} which says that $x^*$ solves the minimization problem $\min_{x\in\Omega_{2}}d(x,\Omega_{1})$ if and only if $x^*=P_2(P_1(x^*))$. 

Fourth, the convergence part of Theorem \bref {thm:Main} above extends, but does not generalize, \cite[Theorem 4]{CheneyGoldstein1959jour} since the setting in \cite[Theorem 4]{CheneyGoldstein1959jour} is a possibly infinite-dimensional real Hilbert space, with possibly unbounded and non-strictly convex subsets, while in the convergence part of Theorem \bref {thm:Main} above we require that the space is a finite-dimensional, that all the subsets (pivot and satellite) are compact, and we also require the satellite subsets to be strictly convex.

\section{Preparatory results}\label{sec:Auxiliary}
In this section we formulate and prove (or refer to known proofs in the literature) several preparatory results which are needed for the proof of Theorem \bref{thm:Main}. Some of the results hold in a setting which is more general than the setting of this theorem. 

\begin{Lem}\label{lem:P(x)-P(y)=x-y}
Suppose that $(Y,\|\cdot\|)$ is a normed space, $\emptyset\neq S\subseteq Y$, and $F:S\to Y$ is $\nu$-firmly nonexpansive for some $\nu>0$. Given $(x,y)\in S^2$, one has $\|F(x)-F(y)\|=\|x-y\|$ if and only if $F(x)-F(y)=x-y$, namely, if and only if the points $x$, $y$, $F(x)$ and $F(y)$ are the four vertices of the possibly degenerate parallelogram $\{x+t(y-x)+s(F(x)-x)\,|\,(s,t)\in [0,1]^2\}$ (see Figures \bref{fig:NonDegenerateParallelogram} and \bref{fig:DegenerateParallelogram}). 
\end{Lem}
\begin{proof}
If $F(x)-F(y)=x-y$, then obviously $\|F(x)-F(y)\|=\|x-y\|$. Conversely, if $\|F(x)-F(y)\|=\|x-y\|$, then since $F$ is $\nu$-firmly nonexpansive, the inequality $\|F(x)-F(y)\|^2\leq \|x-y\|^2-\nu\|(x-y)-(F(x)-F(y))\|^2$ holds. This inequality and the nonnegativity of the norm imply that $0\leq \|(x-y)-(F(x)-F(y))\|^2\leq (1/\nu)(\|F(x)-F(y)\|^2-\|x-y\|^2)=0$. Hence $\|(x-y)-(F(x)-F(y))\|^2=0$ and  $(x-y)-(F(x)-F(y))=0$, as required. 
\end{proof}

\begin{figure}[htb]
\begin{minipage}{0.48\textwidth}
\begin{center}{\includegraphics[trim=103 370 50 100, clip=true, scale=0.6]{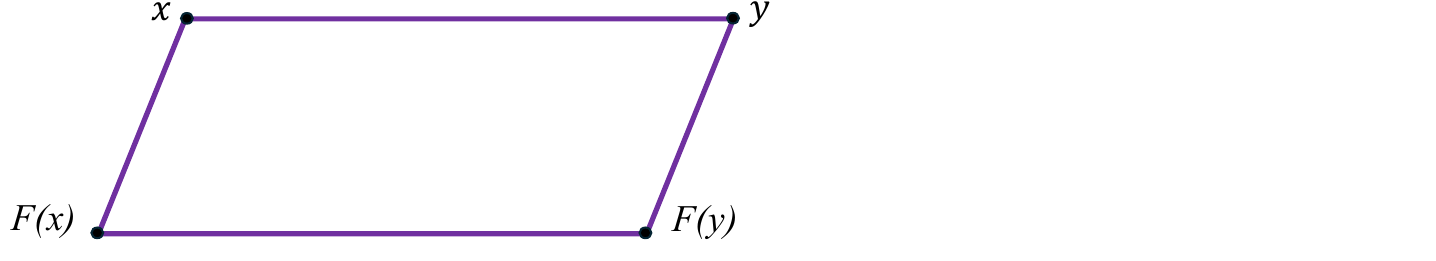}}
\end{center}
 \caption{The setting of Lemma \bref{lem:P(x)-P(y)=x-y}: a nondegenerate parallelogram.}
\label{fig:NonDegenerateParallelogram}
\end{minipage}
\hfill
\begin{minipage}{0.48\textwidth}
\begin{center}{\includegraphics[trim=50 400 50 70, clip=true, scale=0.6]{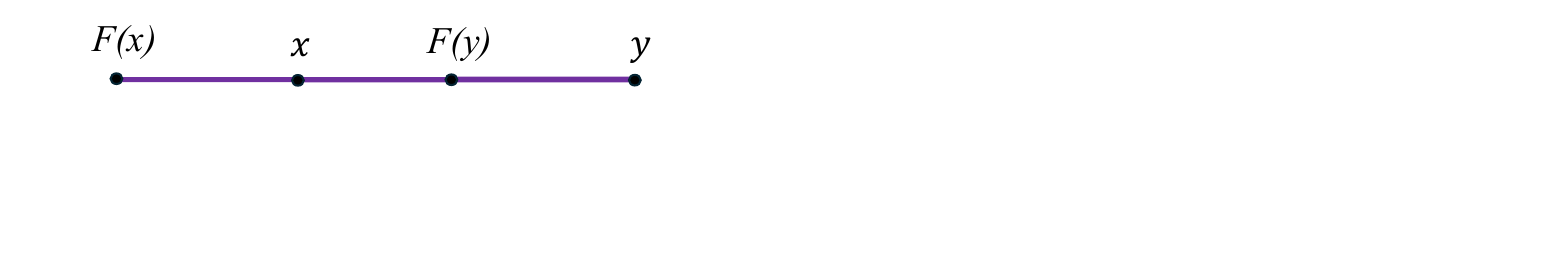}}
\end{center}
 \caption{The settings of Lemma \bref{lem:P(x)-P(y)=x-y}: a degenerate parallelogram.}
\label{fig:DegenerateParallelogram}
\end{minipage}
\end{figure}

\begin{remark}\label{rem:StrictlyNonexpansive}
\begin{enumerate}[(i)]
\item A nonexpansive operator $F:S\to Y$ for which the equality $\|F(x)-F(y)\|=\|x-y\|$ implies $F(x)-F(y)=x-y$ is called in \cite[p. 556]{KoltrachtLancaster1990jour} and \cite[Definition 2]{ElsnerKoltrachtNeumann1992jour} (in the case where $S=Y=\R^n$ for some $n\in\N$) strictly nonexpansive (unfortunately, the term ``strictly nonexpansive'' has other meanings in the literature: see Section \bref{sec:Preliminaries} above). We conclude from Lemma \bref{lem:P(x)-P(y)=x-y} that any  $\nu$-firmly nonexpansive operator ($\nu>0$) is strictly nonexpansive. For an interesting class of strictly nonexpansive operators which is related to Lemma \bref{lem:Rectangle} below see \cite[Definition 1, Proposition 1, Proposition 2, Proposition 8]{DePierroIusem1990jour}. 
\item Since the orthogonal projection $P_S$ onto a nonempty, closed and convex subset $S$ of a real Hilbert space $X$ is 1-firmly nonexpansive, we conclude from Lemma \bref{lem:P(x)-P(y)=x-y} that for all $(x,y)\in X^2$ one has $\|P_S(x)-P_S(y)\|=\|x-y\|$ if and only if $P_S(x)-P_S(y)=x-y$. In particular, if $\|P_S(x)-P_S(y)\|=\|x-y\|$ for some $(x,y)\in X^2$, then $P_S(x)-x=P_S(y)-y$ and so $\|P_S(x)-x\|=\|P_S(y)-y\|$. This is the necessary condition mentioned in \cite[Theorem 3]{CheneyGoldstein1959jour} for equality in the nonexpansiveness inequality $\|P_S(x)-P_S(y)\|\leq\|x-y\|$.  Lemma \bref{lem:P(x)-P(y)=x-y} presents a more general necessary and sufficient condition. 
\end{enumerate}

\end{remark}

Lemma \bref{lem:Rectangle} below is closely related to \cite[Proposition 8]{DePierroIusem1990jour}. Our arguments, however, are different from the ones presented in the proof of \cite[Proposition 8]{DePierroIusem1990jour} and have a more geometric flavour. 
\begin{Lem}\label{lem:Rectangle}
Let $S$ be a nonempty, closed, and convex subset of a real Hilbert space  $(Y,\langle \cdot,\cdot \rangle)$ with norm $\|\cdot\|$, $Y\neq\{0\}$. Let $x,y\in Y$ be arbitrary such that $x\neq y$ and either $x\notin S$ or $y\notin S$. If $\dim(Y)=1$, then $\|P_S(x)-P_S(y)\|<\|x-y\|$. If $\dim(Y)>1$ (possibly $\dim(Y)=\infty$) and $\|P_S(x)-P_S(y)\|=\|x-y\|$, then all the points $x$, $y$, $P_S(x)$ and $P_S(y)$ are distinct, they are located on the same affine plane, the convex hull spanned by them is a rectangle where each of these points is a corner of the rectangle, and $x-y$ is orthogonal to $P_S(y)-y$ (see Figure \bref{fig:rectangle}).
\end{Lem}
\begin{proof}
Since either $x\notin S$ or $y\notin S$, we can assume that $y\notin S$ (the case where $x\notin S$ can be treated in a similar manner). 
First we show that if $\|P_S(x)-P_S(y)\|=\|x-y\|$, then the points $x$, $y$ and $P_S(y)$ are distinct and are not located on the same line. Indeed, under our assumptions, since $\|P_S(x)-P_S(y)\|=\|x-y\|$ and $P_S$ is 1-firmly nonexpansive \cite[Proposition 4.16, p. 74]{BauschkeCombettes2017book}, it follows from Lemma \bref{lem:P(x)-P(y)=x-y} that 
\begin{equation}\label{eq:P_S(x)-P_S(y)=x-y}
P_S(x)-P_S(y)=x-y.
\end{equation}
Since $y\notin S$, we have $y\neq P_S(y)\in S$. If $x=P_S(y)$, then $x\in S$ and hence $x=P_S(x)$.  Thus \beqref{eq:P_S(x)-P_S(y)=x-y} implies that $y-P_S(y)=0$, a contradiction. Hence $x\neq P_S(y)$, and since we already assume that $x\neq y$ and showed above that $y\neq P_S(y)$, we conclude that all the points $x$, $y$ and $P_S(y)$ are distinct. 

Next we show that the points $x$, $y$ and $P_S(y)$ cannot be located on the same line. Indeed, assume by way of contradiction that they are located on the same line $L$. Since $y\neq P_S(y)$ and both $y$ and $P_S(y)$ belong to $L$, we can represent $L$ as $L=\{y+t(P_S(y)-y)\,|\, t\in \R\}$. Since $x\in L$, we have $x=y+t_x(P_S(y)-y)$ for some $t_x\in\R$. Since $x\neq y$ and $x\neq P_S(y)$, we have $t_x\neq 0$ and $t_x\neq 1$. Hence there are two cases to consider:\\

{\bf \noindent  Case 1. Either $t_x<0$ or $t_x\in (0,1)$:} In this case the point $x$ is on the ray $\{P_S(y)+s(y-P_S(y))\,|\, s\in [0,\infty)\}$ emanating from $P_S(y)$ in the direction of $y$, and hence $P_S(x)=P_S(y)$ because the orthogonal projection is a sunny retraction, namely $P_S(z)=P_S(y)$ for all $z$ on the ray emanating from $P_S(y)$ in the direction of $y$: see \cite[p. 17]{GoebelReich1984book} (this fact can be proved directly, using basic properties of orthogonal projections such as the so-called Kolmogorov property mentioned below in \beqref{eq:KolmogorovProperty}; see \cite[Lemma 2.7]{Reich1973jour} for a much more general result). Thus \beqref{eq:P_S(x)-P_S(y)=x-y} implies that $x=y$, which is a contradiction. \\

{\bf \noindent Case 2. $t_x>1$:} In this case either $x\in S$ or $x\notin S$. If $x\in S$, then $x=P_S(x)$, and then \beqref{eq:P_S(x)-P_S(y)=x-y} implies that $P_S(y)=y$, a contradiction. Hence $x\notin S$ and, therefore, $x\neq P_S(x)\in S$. From  \beqref{eq:P_S(x)-P_S(y)=x-y} and the equality $x=y+t_x(P_S(y)-y)$ we have $P_S(x)=x+P_S(y)-y=y+t_x(P_S(y)-y)+P_S(y)-y=y+(t_x+1)(P_S(y)-y)$, namely $P_S(x)\in L$ with $t:=t_x+1$. 

Since $1<t_x<t_x+1$, it follows that $x=y+t_x(P_S(y)-y)$ is between $P_S(y)=y+1\cdot(P_S(y)-y)$ and $P_S(x)=y+(t_x+1)(P_S(y)-y)$, and so the convexity of $S$ implies that $x\in S$. But then $x=P_S(x)$, and hence \beqref{eq:P_S(x)-P_S(y)=x-y} implies that $y=P_S(y)\in S$, a contradiction to what we assumed above. \\

The previous analysis shows that if $\dim(Y)=1$, $x\neq y$ and either $x\notin S$ or $y\notin S$, then the equality $\|P_S(x)-P_S(y)\|=\|x-y\|$ cannot hold because obviously  when $\dim(Y)=1$, then $x$, $y$ and $P_S(y)$ are located on the same line. As a result, since $P_S$ is nonexpansive, we conclude that $\|P_S(x)-P_S(y)\|<\|x-y\|$ in this case.  

From now on we assume that $\dim(Y)\geq 2$ (possibly $\dim(Y)=\infty$), together with the assumptions from the beginning of the proof (that is, $x\neq y$, $\|P_S(x)-P_S(y)\|=\|x-y\|$ and $y\notin S$). In this case, by what we have shown above, we know that $x$, $y$ and $P_S(y)$ are distinct and are not located on the same line. Hence the affine space generated by $x$, $y$ and $P_S(y)$, namely $E:=P_S(y)+\textnormal{span}\{x-P_S(y),y-P_S(y)\}$, is two-dimensional. 

Since $x\neq y$, we have $\|P_S(x)-P_S(y)\|=\|x-y\|>0$, and, therefore, $P_S(x)\neq P_S(y)$. In addition, it must be true that $x\neq P_S(x)$: indeed, if $x=P_S(x)$, then \beqref{eq:P_S(x)-P_S(y)=x-y} implies that $y=P_S(y)\in S$, a contradiction. Since $y\notin S$, we also have $y\neq P_S(x)\in S$, and we conclude that all the four points $x$, $y$,  $P_S(x)$ and $P_S(y)$ are distinct. 

It follows from \beqref{eq:P_S(x)-P_S(y)=x-y} that $P_S(x)=P_S(y)+1\cdot (x-P_S(y))+(-1)\cdot(y-P_S(y))\in E$, and so all the four points $x$, $y$,  $P_S(x)$ and $P_S(y)$ are on the same affine plane $E$. Moreover, from \beqref{eq:P_S(x)-P_S(y)=x-y} and the fact that all the points are different from each other it follows that the nondegenerate interval $[P_S(x),P_S(y)]$ is parallel to the nondegenerate interval $[x,y]$. 

The equality \beqref{eq:P_S(x)-P_S(y)=x-y} implies that $P_S(x)-x=P_S(y)-y$, and hence the non-degenerate interval $[P_S(x),x]$ is parallel to the nondegenerate interval $[P_S(y),y]$. Thus $x$, $y$, $P_S(y)$, and $P_S(x)$ form a parallelogram: see Figure \bref{fig:parallelogram}. 

By the, so-called, Kolmogorov property of orthogonal projections, we have 
\begin{equation}\label{eq:KolmogorovProperty}
\langle x-P_S(x), z-P_S(x)\rangle \leq 0, \quad \forall z\in S.
\end{equation}
In particular, \beqref{eq:KolmogorovProperty} holds with $z:=P_S(y)$, namely the angle $\alpha_1:=\angle xP_S(x)P_S(y)$ between $x-P_S(x)$ and $P_S(y)-P_S(x)$ satisfies $\alpha_1\geq \pi/2$. Similarly (namely by replacing $x$ by $y$ in \beqref{eq:KolmogorovProperty} and letting there $z:=P_S(x)$), $\alpha_2:=\angle yP_S(y)P_S(x)$ satisfies $\alpha_2\geq \pi/2$.  Since $[P_S(x),x]$ is parallel to $[P_S(y),y]$, basic Euclidean plane geometry tells us that $\alpha_1+\alpha_2=\pi$. Thus, it must be that $\alpha_1=\pi/2$, since, otherwise, $\alpha_1>0.5\pi$ and hence $\pi=0.5\pi+0.5\pi<\alpha_1+\alpha_2=\pi$, a contradiction. Therefore, $\alpha_2=\pi-\alpha_1=\pi/2$, and we conclude that the parallelogram $xyP_S(y)P_S(x)$ is actually a rectangle and $x-y$ is orthogonal to $P_S(y)-y$: see Figure \bref{fig:rectangle}.  
\end{proof}

\begin{figure}[t]
\begin{minipage}{1\textwidth}
\begin{center}{\includegraphics[trim=80 485 20 100, clip=true, scale=0.59]{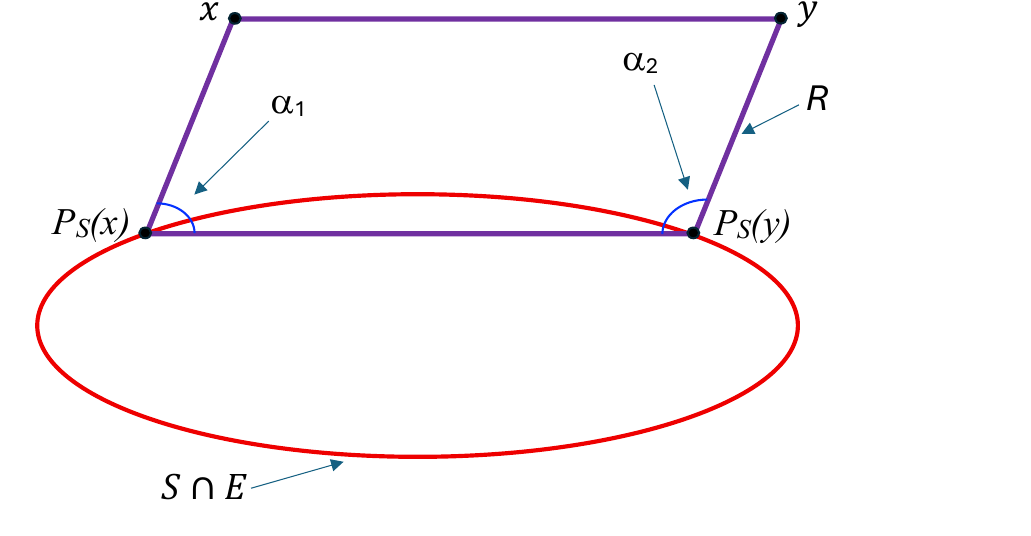}}
\end{center}
 \caption{The setting of Lemma \bref{lem:Rectangle}, where the points $x$, $y$, $P_S(y)$, $P_S(x)$ are the corners of  a parallelogram located on the affine plane $E$.}
\label{fig:parallelogram}
\end{minipage}
\hfill
\begin{minipage}{1\textwidth}
\begin{center}{\includegraphics[trim=80 490 20 70, clip=true, scale=0.59]{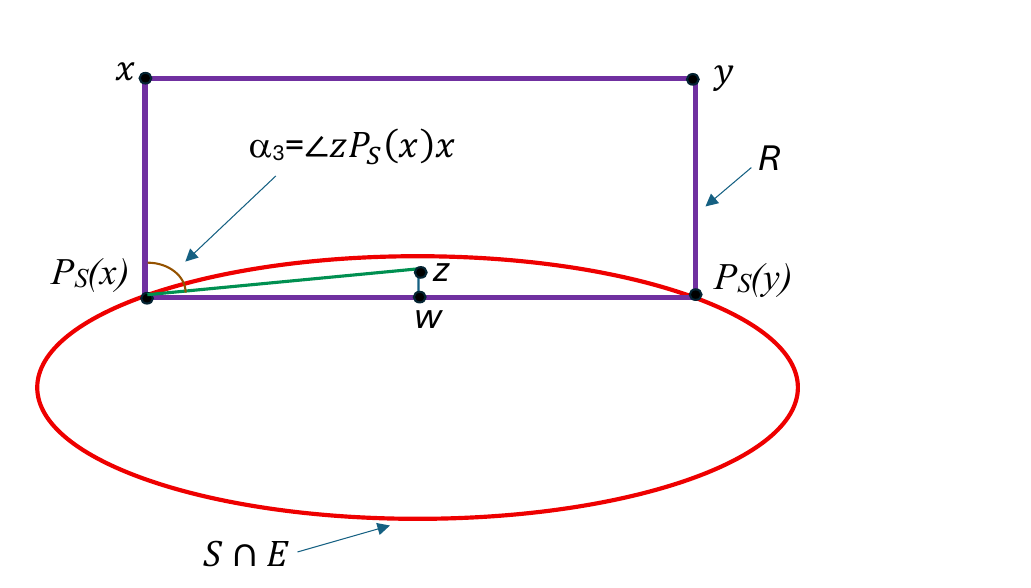}}
\end{center}
 \caption{The settings of Lemma \bref{lem:Rectangle} and Lemma \bref{lem:|P_S(x)-P_S(y)|<|x-y|}, where the points $x$, $y$, $P_S(y)$, $P_S(x)$ are the corners of  a rectangle $R$ located on the affine plane $E$, the point $z$ is an interior point of $R\cap S$, and $\alpha_3:=\angle zP_S(x)x$ satisfies $0<\alpha_3<0.5\pi$.}
\label{fig:rectangle}
\end{minipage}
\end{figure}

The next lemma is known \cite[Lemma 2]{GubinPolyakRaik1967jour}, but is not very well-known, and we present a new proof of it. We also remark that there are issues in the English translation of  \cite[Lemma 2]{GubinPolyakRaik1967jour} (and of \cite{GubinPolyakRaik1967jour} in general), and these issues may have contributed to the fact that \cite[Lemma 2]{GubinPolyakRaik1967jour} is not very known. The issues are that it should be ``strictly convex'' instead of ``strongly convex'' in the line located between (5) and (5') in the formulation of \cite[Lemma 1]{GubinPolyakRaik1967jour} and also in the sixth line of the proof of this lemma (this lemma is used in the proof of \cite[Lemma 2]{GubinPolyakRaik1967jour}), and it should be ``strictly convex'' instead of ``strongly convex'' also in the line located between (6) and (6') in the formulation \cite[Lemma 2]{GubinPolyakRaik1967jour}.

\begin{Lem}\label{lem:|P_S(x)-P_S(y)|<|x-y|}
Let $S$ be a nonempty, closed, and strictly convex subset of a real Hilbert space $(Y,\langle \cdot,\cdot \rangle)$ with norm $\|\cdot\|$, $Y\neq\{0\}$. Let $x,y\in Y$ be arbitrary such that $x\neq y$ and either $x\notin S$ or $y\notin S$. Then $\|P_S(x)-P_S(y)\|<\|x-y\|$. Stated differently, for all $x$ and $y$ in $Y$ which satisfy $\|P_S(x)-P_S(y)\|=\|x-y\|$ it must hold that either $x=y$ or both $x\in S$ and $y\in S$. 
\end{Lem}
\begin{proof}
The case $\dim(Y)=1$ follows from Lemma \bref{lem:Rectangle}. Hence from now on $\dim(Y)\geq 2$ (possibly $\dim(Y)=\infty$). Assume for a contradiction that $\|P_S(x)-P_S(y)\|=\|x-y\|$. Then Lemma \bref{lem:Rectangle} implies that all the four points $x$, $y$,  $P_S(x)$ and $P_S(y)$ are distinct and each of them is a corner of a rectangle $R$ in which $x-y$ is orthogonal to $P_S(y)-y$ (and hence $P_S(y)-P_S(x)$ is orthogonal to $x-P_S(x)$). See Figure \bref{fig:rectangle}. 

Since $S$ is strictly convex and $P_S(x)\in S$ and $P_S(y)\in S$, any point in the open line segment $(P_S(x),P_S(y))$ is in the interior of $S$. In particular, $w:=0.5(P_S(x)+P_S(y))$ is in the interior of $S$ and hence there is some $r\in (0,\|x-P_S(x)\|)$ such that the open ball of center $w$ and radius $r$ is contained in $S$. In particular, the point $z:=w+(0.5r/\|x-P_S(x)\|)(x-P_S(x))$ is in $S$, and it is also strictly inside $R$ since $r<\|x-P_S(x)\|$ and $z=P_S(x)+0.5(P_S(y)-P_S(x))+(0.5r/\|x-P_S(x)\|)(x-P_S(x))$ and $R=\{P_S(x)+\eta_1(P_S(y)-P_S(x))+\eta_2(x-P_S(x))\,|(\eta_1,\eta_2)\in [0,1]^2\}$ (and the interior of $R$ is $\{P_S(x)+\eta_1(P_S(y)-P_S(x))+\eta_2(x-P_S(x))\,|(\eta_1,\eta_2)\in (0,1)^2\}$). This fact and the fact that $P_S(y)-P_S(x)$ is orthogonal to $x-P_S(x)$ imply that 
\begin{multline*}
\langle x-P_S(x),z-P_S(x)\rangle\\
=\frac{1}{2}\langle x-P_S(x),P_S(y)-P_S(x)\rangle+\frac{r}{2\|x-P_S(x)\|}\langle x-P_S(x),x-P_S(x)\rangle=\frac{1}{2}r\|x-P_S(x)\|.
\end{multline*}
Hence, if we let $\alpha_3:=\angle zP_S(x)x$, i.e., the angle between $z-P_S(x)$ and $x-P_S(x)$, then $\cos(\alpha_3)=\langle x-P_S(x),z-P_S(x)\rangle/(\|x-P_S(x)\|\|z-P_S(x)\|)=0.5r/\|z-P_S(x)\|>0$. Since 
\begin{equation*}
\|z-P_S(x)\|=\sqrt{\|z-w\|^2+\|w-P_S(x)\|^2}=\sqrt{(0.5r)^2+\|0.5(P_S(y)-P_S(x)\|^2}>\frac{r}{2}, 
\end{equation*}
we have $\cos(\alpha_3)=0.5r/\|z-P_S(x)\|<1$. Thus, $0<\alpha_3<\pi/2$ (see Figure \bref{fig:rectangle}), a contradiction to \beqref{eq:KolmogorovProperty} which holds since $z\in S$. The contradiction shows that the equality $\|P_S(x)-P_S(y)\|=\|x-y\|$ is impossible, and since $P_S$ is nonexpansive we conclude that $\|P_S(x)-P_S(y)\|<\|x-y\|$.
\end{proof}

\begin{Lem}\label{lem:Existence_Min}
In the setting of Algorithm \bref{alg:MSGBAP}, if, either $\Omega_{m+1}$ is bounded or it is unbounded and $\Omega_q$ is bounded for some $q\in \{1,2,\ldots,m\}$, then there exists a solution to \beqref{eq:MinProb} in $\Omega_{m+1}$.   
\end{Lem}
\begin{proof}
The real-valued functional $\Phi$ defined in \beqref{eq:Phi} is continuous and convex over $X$, as a sum of squares of nonnegative, convex and continuous functions. Hence (as follows from \cite[Corollary 3.9, p. 61]{Brezis2011book}), it is lower semicontinuous in the weak topology defined on $X$.

Now, if the nonempty, closed and convex subset $\Omega_{m+1}$ is bounded, then $\Omega_{m+1}$ is weakly compact since $X$ is reflexive \cite[Corollary 3.22, p. 71]{Brezis2011book} and hence, by the Weirestrass Extreme Value Theorem, $\Phi$ attains a minimum over $\Omega_{m+1}$, namely there exists a solution to \beqref{eq:MinProb}. If $\Omega_{m+1}$ is unbounded, then since we assume that  $\Omega_q$ is bounded for some $q\in \{1,\ldots,m\}$, we have that $\|P_q\|$ is bounded over $X$ (because if $x\in X$ is arbitrary, and $v\in \Omega_q$ is fixed independently of $x$, then the nonexpansivity of $P_q$ and the triangle inequality imply that $\|P_q(x)\|\leq \|P_q(x)-P_q(v)\|+\|P_q(v)\|\leq \text{diam}(\Omega_q)+\|P_q(v)\|<\infty$).

Hence, the triangle inequality implies that $\infty=\lim_{\|x\|\to\infty, x\in\Omega_{m+1}}(\|x\|-\|P_q(x)\|)\leq \lim_{\|x\|\to\infty, x\in\Omega_{m+1}}\|x-P_q(x)\|$. Therefore, $\infty=\lim_{\|x\|\to\infty, x\in \Omega_{m+1}}0.5\beta_q\|x-P_q(x)\|^2\leq \lim_{\|x\|\to\infty, x\in \Omega_{m+1}}\Phi(x)$ (also because we assume that $\beta_i>0$ for all $i\in\{1,2,\ldots,m\}$: see \beqref{eq:Phi}). Thus, $\Phi$ is coercive, convex and continuous (in particular, lower semicontinuous), and so we can use \cite[Corollary 3.23, p. 71]{Brezis2011book} to conclude that $\Phi$ has a minimizer in $\Omega_{m+1}$. Hence, \beqref{eq:MinProb} has a solution.
\end{proof}

\begin{Lem}\label{lem:FixedPointMin}
In the setting of Algorithm \bref{alg:MSGBAP}, denote by $\textnormal{Argmin}_{\Omega_{m+1}}(\Phi)$ the set of solutions to \beqref{eq:MinProb}, and denote by $\textnormal{Fix}_S(G_{\lambda,\gamma})$ the set of fixed points of the operator $G_{\lambda,\gamma}$ over a subset $S$ of $X$ for each $\lambda\neq 0$ and $\gamma>0$, where $G_{\lambda,\gamma}$ is defined in \beqref{eq:G_lambda_gamma}. Then $\textnormal{Argmin}_{\Omega_{m+1}}(\Phi)=\textnormal{Fix}_X(G_{\lambda,\gamma})=\textnormal{Fix}_{\Omega_{m+1}}(G_{\lambda,\gamma})$ for each $\lambda\neq 0$ and each $\gamma>0$, and therefore, $x$ solves  \beqref{eq:MinProb} if and only if $x$ is a fixed point of  $G_{\lambda,\gamma}$ for each $\lambda\neq 0$ and each $\gamma>0$, if and only if $x$ is a fixed point of $G_{\lambda,\gamma}$ for some $\lambda\neq 0$ and some $\gamma>0$ (and $x$ also belongs to $\Omega_{m+1}$).
\end{Lem}
\begin{proof}
Part of this assertion essentially appears in \cite{CombettesBondon1999jour} (in a somewhat different form), and for the sake of completeness we provide its proof here. We first observe that if $x=G_{\lambda,\gamma}(x)$ for some $x\in X$, $\lambda\neq 0$ and $\gamma\in \R$, then the following equality holds: $x=(1-\lambda)x+\lambda P_{m+1}((1-\gamma)x+\gamma\sum_{i=1}^m \beta_i P_i(x))$. Hence, simple arithmetic implies that $x=P_{m+1}((1-\gamma)x+\gamma\sum_{i=1}^m \beta_i P_i(x))=G_{1,\gamma}(x)$ and $x\in \Omega_{m+1}$ because $P_{m+1}$ maps $X$ onto $\Omega_{m+1}$. Conversely, if  $x=G_{1,\gamma}(x)=P_{m+1}((1-\gamma)x+\gamma\sum_{i=1}^m \beta_i P_i(x))$ for some $x\in X$ and $\gamma\in\R$, then obviously $x\in \Omega_{m+1}$, and simple arithmetic implies that $x=(1-\lambda)x+\lambda x=(1-\lambda)x+\lambda P_{m+1}((1-\gamma)x+\gamma\sum_{i=1}^m \beta_i P_i(x))=G_{\lambda,\gamma}(x)$ for all $\lambda\in\R$ (even for $\lambda=0$). Therefore, $\textnormal{Fix}_X(G_{\lambda,\gamma})=\textnormal{Fix}_{\Omega_{m+1}}(G_{\lambda,\gamma})=\textnormal{Fix}_{X}(G_{1,\gamma})$ for all $\lambda\neq 0$ and all $\gamma\in\R$. 

Now denote $f_i(x):=\frac{1}{2}d(x,\Omega_i)^2$ for all $i\in\{1,\ldots,m\}$ and all $x\in X$. This function is continuously differentiable and its gradient satisfies $\nabla f_i=\Id-P_{i}$ (see \cite[Lemma 2.2.27, p. 79]{Cegielski2012book}). Thus, the function $f:X\to \R$ defined by $f(x):=\frac{1}{2}\sum_{i=1}^m \beta_i d(x,\Omega_i)^2$ for all $x\in X$, namely, $f=\sum_{i=1}^m \beta_i f_i$, is continuously differentiable. Since $\sum_{i=1}^m\beta_i=1$ we have $\nabla f=\sum_{i=1}^m \beta_i\nabla f_i=\Id-\sum_{i=1}^m\beta_iP_{i}$. 

From \cite[Proposition 3]{CombettesBondon1999jour} we have $\textnormal{Argmin}_{\Omega_{m+1}}(\Phi)=\textnormal{Fix}_X(P_{m+1}(\Id-\gamma\nabla f))$ for each $\gamma>0$. However, the previous lines imply that $P_{m+1}(\Id-\gamma\nabla f)=P_{m+1}((1-\gamma)\Id+\gamma\sum_{i=1}^m \beta_i P_i)=G_{1,\gamma}$, and we conclude from this fact, as well as from the previous lines and the previous paragraph, the equality $\textnormal{Argmin}_{\Omega_{m+1}}(\Phi)=\textnormal{Fix}_X(G_{1,\gamma})=\textnormal{Fix}_X(G_{\lambda,\gamma})=\textnormal{Fix}_{\Omega_{m+1}}(G_{\lambda,\gamma})$ for each $\lambda\neq 0$ and each $\gamma>0$.
\end{proof}
 
\begin{Lem}\label{lem:|G(x)-G(y)|<|x-y|}
In the setting of Algorithm \bref{alg:MSGBAP}, suppose that $\lambda \in [0,1]$ and $\gamma\in [0,2]$, and let $G_{\lambda,\gamma}:X\to X$ be defined by \beqref{eq:G_lambda_gamma}. Then $G_{\lambda,\gamma}$ is nonexpansive over $X$ and maps $\Omega_{m+1}$ to itself. Suppose further that $\Omega_s$ is strictly convex for some $s\in\{1,2,\ldots,m\}$ and $\Omega_s\cap\Omega_{m+1}=\emptyset$, and assume also that $\lambda\in (0,1]$ and $\gamma\in (0,2)$. Then $G_{\lambda,\gamma}$ is Edelstein-contractive over $\Omega_{m+1}$, has at most one fixed point in $X$, and this fixed point - if it exists - is in $\Omega_{m+1}$. If, in addition, either $\Omega_{m+1}$ is bounded or $\Omega_{m+1}$ is unbounded and $\Omega_q$ is bounded for some $q\in \{1,2,\ldots,m\}$, then $G_{\lambda,\gamma}$ has a unique fixed point in $X$ and this fixed point is in $\Omega_{m+1}$.   
\end{Lem}
\begin{proof}
Let $T:=\sum_{i=1}^m\beta_i P_i$ be the weighted sum operator of the operators $P_i$, $i\in \{1,2,\ldots,m\}$, and let  $T_{\gamma}:=(1-\gamma)\Id+\gamma T$ be its $\gamma$-relaxation whenever $\gamma\in [0,2]$. The triangle inequality, the definition of $G_{\gamma,\lambda}$ (see \beqref{eq:G_lambda_gamma}), the nonexpansivity of  $P_{m+1}$, the fact that $(\beta_i)_{i=1}^m$ is a positive weight vector and the fact that $\lambda\in [0,1]$, imply, for each $x\in X$ and $y\in X$, that 
\begin{multline}\label{eq:G_nonexpansive}
\|G_{\lambda,\gamma}(x)-G_{\lambda,\gamma}(y)\|=\Big\|\Big((1-\lambda)x+\lambda P_{m+1}((1-\gamma)x+\gamma\sum_{i=1}^m \beta_i P_i(x))\Big)\\
-\Big((1-\lambda)y+\lambda P_{m+1}((1-\gamma)y+\gamma\sum_{i=1}^m \beta_i P_i(y))\Big)\Big\|\\
\leq (1-\lambda)\|x-y\|+\lambda \|P_{m+1}((1-\gamma)x+\gamma\sum_{i=1}^m \beta_i P_i(x))-P_{m+1}((1-\gamma)y+\gamma\sum_{i=1}^m \beta_i P_i(y))\|\\
\leq (1-\lambda)\|x-y\|+\lambda\|((1-\gamma)x+\gamma\sum_{i=1}^m \beta_i P_i(x))-((1-\gamma)y+\gamma\sum_{i=1}^m \beta_i P_i(y))\|\\
=(1-\lambda)\|x-y\|+\lambda\|T_{\gamma}(x)-T_{\gamma}(y)\|.
\end{multline}
Since the orthogonal projection onto a nonempty, closed and convex subset is firmly nonexpansive \cite[Theorem 2.2.21(iii) on p. 76 and Theorem 2.2.10(v) on p. 70]{Cegielski2012book}, and since a positively weighted sum (namely, with a positive weight vector) of firmly nonexpansive operators is also firmly nonexpansive \cite[Corollary 2.2.20, p. 75]{Cegielski2012book}, we conclude that the operator $T=\sum_{i=1}^m\beta_i P_i$ is firmly nonexpansive. Hence \cite[Theorem 2.2.20(ii)]{Cegielski2012book} implies that its relaxation $T_{\gamma}$ is nonexpansive for all $\gamma\in [0,2]$. Thus from \beqref{eq:G_nonexpansive} and the fact that $\lambda\in [0,1]$ we conclude that $G_{\lambda,\gamma}$ is nonexpansive.

Now let $x\in\Omega_{m+1}$ be arbitrary. The definition of $G_{\lambda,\gamma}$ (see \beqref{eq:G_lambda_gamma}), the assumption that $\lambda\in [0,1]$, the fact that $P_{m+1}(z)\in \Omega_{m+1}$ for every $z\in X$, and the convexity of $\Omega_{m+1}$, all imply that $G_{\lambda,\gamma}(x)$ is a convex combination of $x$ and $P_{m+1}(z)\in \Omega_{m+1}$ for $z:=(1-\gamma)x+\gamma\sum_{i=1}^m \beta_i P_i(x)$. Hence $G_{\lambda,\gamma}(x)\in \Omega_{m+1}$ for each $x\in\Omega_{m+1}$ and therefore $G_{\lambda,\gamma}$ maps $\Omega_{m+1}$ to itself. 

Now we want to show that under certain assumptions $G_{\lambda,\gamma}$ is Edelstein-contractive over $\Omega_{m+1}$. Before recalling these assumptions, we observe that from the calculation presented in \cite[Top of p. 71]{Cegielski2012book} (where we take $\gamma$ instead of $\lambda$ there) we have for all $\gamma\in [0,2]$ and all $(x,y)\in X^2$,
\begin{equation}\label{T_gamma_Inequality}
\|T_{\gamma}(x)-T_{\gamma}(y)\|^2\leq \gamma(2-\gamma)\|T(x)-T(y)\|\|x-y\|+(1-\gamma)^2\|x-y\|^2 .
\end{equation}
(In fact, the calculation done in \cite[Top of p. 71]{Cegielski2012book} and also \beqref{T_gamma_Inequality} hold for any firmly nonexpansive operator $T$ and not just for $T=\sum_{i=1}^m\beta_i P_i$.) Therefore \beqref{T_gamma_Inequality}, the definition of $T$, the fact that $(\beta_i)_{i=1}^m$ is a weight vector, and the triangle inequality, all imply that 
\begin{multline}\label{eq:T_gamma^2}
\|T_{\gamma}(x)-T_{\gamma}(y)\|^2\leq
\gamma(2-\gamma)\left\|\sum_{i=1}^m\beta_i(P_i(x)-P_i(y))\right\|\cdot\|x-y\|+(1-\gamma)^2\|x-y\|^2\\
\leq \gamma(2-\gamma)\|x-y\|\sum_{i=1}^m\beta_i\|P_i(x)-P_i(y)\|+(1-\gamma)^2\|x-y\|^2.
\end{multline}

Now assume that $\Omega_s$ is strictly convex for some $s\in\{1,2,\ldots,m\}$, that $\Omega_s\cap\Omega_{m+1}=\emptyset$, that $\lambda\in (0,1]$ and that $\gamma\in (0,2)$. Take any pair of points $(x,y)\in \Omega_{m+1}^2$ such that $x\neq y$. Then both $x\notin \Omega_s$ and $y\notin \Omega_s$, and hence Lemma \bref{lem:|P_S(x)-P_S(y)|<|x-y|} (applied to $S:=\Omega_s$), the nonexpansivity of the orthogonal projections $P_i$ for every  $i\in\{1,2,\ldots,m\}$, the fact that $(\beta_i)_{i=1}^m$ is a positive weight vector and hence $\beta_s>0$, and the assumptions that $\Omega_s$ is strictly convex, $x\neq y$, $\gamma\in (0,2)$ (which implies that $\gamma(2-\gamma)>0$) and \beqref{eq:T_gamma^2}, all imply that 
\begin{equation}\label{eq:EdelsteinContractiveT_gamma}
\begin{split}
\|T_{\gamma}(x)-T_{\gamma}(y)\|^2
&\leq\gamma(2-\gamma)\|x-y\|\sum_{i=1}^m\beta_i\|P_i(x)-P_i(y)\|+(1-\gamma)^2\|x-y\|^2\\
&=\gamma(2-\gamma)\|x-y\|\sum_{i\neq s}\beta_i\|P_i(x)-P_i(y)\|\\
&+\gamma(2-\gamma)\beta_s\|x-y\|\|P_s(x)-P_s(y)\|+(1-\gamma)^2\|x-y\|^2\\
&<\gamma(2-\gamma)\|x-y\|\sum_{i\neq s}\beta_i\|x-y\|+\gamma(2-\gamma)\beta_s\|x-y\|^2\\
&+(1-\gamma)^2\|x-y\|^2\\
&=\gamma(2-\gamma)\|x-y\|^2\sum_{i=1}^m\beta_i+(1-\gamma)^2\|x-y\|^2\\
&=\|x-y\|^2.
\end{split}
\end{equation}
Thus $\|T_{\gamma}(x)-T_{\gamma}(y)\|<\|x-y\|$, and when we combine this inequality with  \beqref{eq:G_nonexpansive} and the assumption that $\lambda>0$, we see that $G_{\lambda,\gamma}$ is Edelstein contractive under the assumptions imposed before \beqref{eq:EdelsteinContractiveT_gamma}. 

As for the fixed points of $G_{\lambda,\gamma}$ over $X$, Lemma \bref{lem:FixedPointMin} tells us that any fixed point of $G_{\lambda,\gamma}$, if it exists, belongs to $\Omega_{m+1}$. Now, if $x$ and $y$ are two distinct fixed points of $G_{\lambda,\gamma}$ in $X$, then they belong to $\Omega_{m+1}$ and $\|x-y\|=\|G_{\lambda,\gamma}(x)-G_{\lambda,\gamma}(y)\|<\|x-y\|$ since  $G_{\lambda,\gamma}$ is Edelstein-contractive over $\Omega_{m+1}$ as we showed above. This contradiction shows that  $G_{\lambda,\gamma}$ has at most one fixed point in $\Omega_{m+1}$ and hence (Lemma \bref{lem:FixedPointMin}) at most one fixed point in $X$. 

It remains to consider the issue of existence of a fixed point of $G_{\lambda,\gamma}$ under the further assumption that  either $\Omega_{m+1}$ is bounded or $\Omega_{m+1}$ is unbounded but $\Omega_q$ is bounded for some $q\in \{1,2,\ldots,m\}$. This is an immediate consequence of Lemma \bref{lem:Existence_Min} and Lemma \bref{lem:FixedPointMin}. 
\end{proof}

The next lemma is recent and is essential for the convergence analysis.
\begin{Lem}\label{lem:Lemma_27_CensorMansourReem2024jour} (a particular case of \cite[Lemma 27]{CensorMansourReem2024jour}, with a different notation) Under the conditions of Algorithm \bref{alg:MSGBAP}, suppose that $X$ is finite-dimensional. Then for each $i\in\{1,2,\ldots,m\}$ and each $\rho>0$ the sequence $(\wh{P}_{i,k})_{k=0}^{\infty}$ converges uniformly, over $B[0,\rho]$, to the orthogonal projection $P_{i}$ onto $\Omega_i$.
\end{Lem}
We finish this section with the following well-known lemma whose proof is just an exercise: see \cite[Ex. 8, p. 134]{Munkres2014book}. 
\begin{Lem}\label{lem:g_k(x^k)}
Suppose that $(X_1,\mathcal{T}_1)$ is a topological space and $(X_2,d_2)$ is a metric space. Assume also that $(g_k)_{k=0}^{\infty}$ is a sequence of continuous functions from $X_1$ to $X_2$ which converges uniformly to a function $g:X_1\to X_2$. Suppose also that $(z^k)_{k=0}^{\infty}$ is a sequence in $X_1$ which converges to $z\in X_1$. Then $\lim_{k\to\infty}g_k(z^k)=g(z)$.
\end{Lem}

\section{Existence, uniqueness and convergence analysis for the BAT algorithm}\label{sec:Convergence-analysis}

In this section we present the analysis needed for proving our main existence, uniqueness and convergence result, namely, Theorem \bref{thm:Main} formulated above. The theorem itself is proved at the end of the section. 

\begin{Lem}\label{lem:x^k_is_bounded}
In the setting of Algorithm \bref{alg:MSGBAP}, for each $k\in\N\cup\{0\}$ the function $G_k$ maps $\Omega_{m+1}$ into itself. In particular, if $x^0\in \Omega_{m+1}\subseteq B[0,\rho]$ for some $\rho>0$, then any sequence $(x^k)_{k\in\N\cup\{0\}}$ generated by Algorithm \bref{alg:MSGBAP} is contained in $\Omega_{m+1}$ and hence bounded.
\end{Lem}
\begin{proof}
Fix $k\in\N\cup\{0\}$, let $x\in \Omega_{m+1}$ and let $y:=(1-\gamma_k)x+\sum_{i=1}^m\beta_i\wh{P}_{i,k}(x)\in X$. Then $P_{m+1}(y)\in \Omega_{m+1}$ and hence \beqref{eq:G_k} implies that $G_k(x)=(1-\lambda_k)x+\lambda_kP_{m+1}(y)\in \Omega_{m+1}$ as a convex combination of two elements in the convex set $\Omega_{m+1}$. Thus, $G_k$ maps $\Omega_{m+1}$ into itself for each $k\in\N\cup\{0\}$. Now assume that $x^0\in\Omega_{m+1}$. Then $x^1=G_0(x^0)\in \Omega_{m+1}$, by \beqref{eq:GBAP} and by what we have just proved. Thus, induction, as well as \beqref{eq:GBAP}, show that $x^k\in\Omega_{m+1}$ for each $k\in\N\cup\{0\}$. Since $\Omega_{m+1}$ is contained in $B[0,\rho]$, we conclude that $(x^k)_{k\in\N\cup\{0\}}$ is bounded. 
\end{proof}

\begin{Lem}\label{lem:UniformConvergence}
In the setting of Algorithm \bref{alg:MSGBAP}, suppose that $X$ is finite-dimensional. If there is an infinite subset $K$ of $\N\cup\{0\}$ such that the sequences of scalars $(\lambda_k)_{k\in K}$ and $(\gamma_k)_{k\in K}$ converge to the real numbers $\lambda\in [0,1]$ and $\gamma\in  [0,1]$, respectively, and if there is $\rho>0$ such that $\Omega_i\subseteq B[0,\rho]$ for all $i\in \{1,2,\ldots,m\}$, then the sequence of operators $(G_k)_{k\in K}$ defined by \beqref{eq:G_k} for all $k\in K$ converges uniformly to $G_{\lambda,\gamma}:=(1-\lambda)Id+\lambda P_{m+1}((1-\gamma)Id+\gamma\sum_{i=1}^m \beta_iP_{i})$ over $B[0,\rho]$.
\end{Lem}
\begin{proof}
Let $\varepsilon>0$ be arbitrary. Since $P_i$ is nonexpansive for all $i\in\{1,2,\ldots,m+1\}$, these operators are bounded on the bounded set $B[0,\rho]$. Hence there is a positive number $M$ such that $\|P_i(x)\|<M$ for all $x\in B[0,\rho]$ and all $i\in\{1,2,\ldots,m+1\}$.

Since $\lambda=\lim_{k\in K}\lambda_k$ and $\gamma=\lim_{k\in K}\gamma_k$, there is $k_0\in K$ such that $|\lambda_k-\lambda|<(1/6)\varepsilon/(\rho+M)$ and $|\gamma_k-\gamma|<(1/6)\varepsilon/(\rho+M)$ for every $k_0\leq k\in K$.  Because $(\tau_{i,k})_{k=0}^{\infty}$ is a sequence of steering parameters for each $i\in\{1,\ldots,m+1\}$ and because $X$ is finite-dimensional, it follows from Lemma \bref{lem:Lemma_27_CensorMansourReem2024jour} that for each $i\in\{1,\ldots,m\}$ the sequence $(\wh{P}_{i,k})_{k=0}^{\infty}$ converges uniformly, over $B[0,\rho]$, to the orthogonal projection $P_{i}$ onto $\Omega_i$. Thus, its subsequence $(\wh{P}_{i,k})_{k\in K}$ also converges uniformly, over $B[0,\rho]$, to $P_{i}$ for each $i\in\{1,\ldots,m\}$, and hence there is $k_0\leq k_1\in K$ such that $\|\wh{P}_{i,k}(z)-P_i(z)\|<(1/6)\varepsilon$ for all $k_1\leq k\in K$ and all $i\in\{1,2,\ldots,m\}$ and all $z\in B[0,\rho]$.

These facts, as well as the nonexpansiveness of $P_{m+1}$, the fact that $\lambda_k\in [0,1]$ and $\gamma_k\in [0,1]$ for all $k\in K$, the fact that $(1-\gamma)x+\gamma\sum_{i=1}^m\beta_i P_i(x)\in B[0,\rho]$ for each $x\in B[0,\rho]$ (due to the convexity of $B[0,\rho]$, the inclusions $\Omega_i\subseteq B[0,\rho]$, $i\in\{1,2,\ldots,m\}$, because $\gamma\in [0,1]$ and because $(\beta_i)_{i=1}^m$ is a weight vector),  and simple algebra, all imply that for every $x\in B[0,\rho]$ and every $k_1\leq k\in K$,  
\begin{multline*}
\|G_k(x)-G_{\lambda,\gamma}(x)\|=\|((1-\lambda_k)x+\lambda_k P_{m+1}((1-\gamma_k)x+\gamma_k\sum_{i=1}^m\beta_i \wh{P}_{i,k}(x)))\\
-((1-\lambda)x+\lambda P_{m+1}((1-\gamma)x+\gamma\sum_{i=1}^m\beta_i P_i(x)))\|\\
\leq |\lambda_k-\lambda|\|x\|+\lambda_k\|P_{m+1}((1-\gamma_k)x+\gamma_k\sum_{i=1}^m\beta_i\wh{P}_{i,k}(x))-P_{m+1}((1-\gamma)x+\gamma\sum_{i=1}^m\beta_i P_i(x))\|\\
+\|\lambda_k P_{m+1}((1-\gamma)x+\gamma\sum_{i=1}^m\beta_i P_i(x))-\lambda P_{m+1}((1-\gamma)x+\gamma\sum_{i=1}^m\beta_i P_i(x))\|\\
\leq |\lambda_k-\lambda|\|x\|+\lambda_k\|((1-\gamma_k)x+\gamma_k\sum_{i=1}^m\beta_i\wh{P}_{i,k}(x))-((1-\gamma)x+\gamma\sum_{i=1}^m\beta_i P_i(x))\|+\\
|\lambda_k-\lambda|\|P_{m+1}((1-\gamma)x+\gamma\sum_{i=1}^m\beta_i P_i(x)))\|\\
< \frac{(1/6)\varepsilon\rho}{\rho+M}+\|\gamma_k-\gamma\|\|x\|+\|\gamma_k\sum_{i=1}^m\beta_i\wh{P}_{i,k}(x)-\gamma\sum_{i=1}^m\beta_i P_i(x)\|+\frac{(1/6)\varepsilon M}{\rho+M}\\
< \frac{1}{6}\varepsilon+\frac{(1/6)\varepsilon\rho}{\rho+M}+\gamma_k\|\sum_{i=1}^m\beta_i\wh{P}_{i,k}(x)-\sum_{i=1}^m\beta_i P_i(x)\|+|\gamma_k-\gamma|\sum_{i=1}^m\beta_i\|P_i(x)\|+\frac{1}{6}\varepsilon\\
< \frac{3}{6}\varepsilon+\sum_{i=1}^m\beta_i\|\wh{P}_{i,k}(x)-P_i(x)\|+\frac{(1/6)\varepsilon M}{\rho+M}<\frac{3}{6}\varepsilon+\frac{1}{6}\varepsilon+\frac{1}{6}\varepsilon<\varepsilon.
\end{multline*}
Since $\varepsilon$ was an arbitrary positive number we conclude that $G_{\lambda,\gamma}=\lim_{k\to \infty, k\in K}G_k$ uniformly on $B[0,\rho]$, as claimed. 
\end{proof}

\begin{Lem}\label{lem:ExistenceUniqueness}
 In the setting of Algorithm \bref{alg:MSGBAP}, if either $\Omega_{m+1}$ is bounded or $\Omega_{m+1}$ is unbounded and $\Omega_q$ is bounded for some $q\in \{1,2,\ldots,m\}$, then \beqref{eq:MinProb} has at least one solution. If, in addition, $\Omega_s$ is strictly convex for some $s\in \{1,2,\ldots,m\}$ and also $\Omega_s\cap \Omega_{m+1}=\emptyset$, then \beqref{eq:MinProb} has a unique solution $x^*$, and $x^*$ has the property that it is the unique fixed point of each of the functions $G_{\lambda,\gamma}:X\to X$, $\lambda\in (0,1]$, $\gamma\in (0,2)$ defined in \beqref{eq:G_lambda_gamma}.
\end{Lem}
\begin{proof}
Existence of a solution to \beqref{eq:MinProb} in $\Omega_{m+1}$ when either $\Omega_{m+1}$ is bounded or when it is unbounded and $\Omega_q$ is bounded for some $q\in\{1,2,\ldots,m\}$ follows from Lemma \bref{lem:Existence_Min}. Uniqueness of the solution to \beqref{eq:MinProb} when there is some $s\in\{1,2,\ldots,m\}$ such that $\Omega_s$ is strictly convex and $\Omega_s\cap\Omega_{m+1}=\emptyset$, as well as the fact that this unique solution is a fixed point of $G_{\lambda,\gamma}$ for each $\lambda\in (0,1]$ and $\gamma\in (0,2)$, follow from the combination of Lemma \bref{lem:FixedPointMin} and Lemma \bref{lem:|G(x)-G(y)|<|x-y|}.
\end{proof}

\begin{Lem}\label{lem:subsequence_x^*}
Suppose that all the conditions mentioned in Theorem \bref{thm:Main} hold. Then there exists an infinite subset $K$ of $\N\cup\{0\}$ and a subsequence $(x^k)_{k\in K}$ of $(x^k)_{k=0}^{\infty}$ which converges to $x^*$. Moreover, $\lim_{k\to\infty, k\in K}\|x^{k+1}-x^k\|=0$.
\end{Lem}
\begin{proof}
Lemma \bref{lem:ExistenceUniqueness} ensures the existence and uniqueness of the solution $x^*$ to \beqref{eq:MinProb}, as well as the fact that it is the unique fixed point of $G_{\lambda,\gamma}$ for each of the functions $G_{\lambda,\gamma}:X\to X$, $\lambda\in (0,1]$, $\gamma\in (0,2)$ defined in \beqref{eq:G_lambda_gamma}. 

Now assume for a contradiction that no subsequence of $(x^k)_{k=0}^{\infty}$ converges to $x^*$. Then there is an $\varepsilon>0$ such that only finitely many elements of the sequence  $(x^k)_{k=0}^{\infty}$ are in the open ball $B(x^*,\varepsilon)$, and so there is $k_0\in\N\cup\{0\}$ such that $x^k\notin B(x^*,\varepsilon)$ for all $k_0\leq k\in\N\cup\{0\}$. We claim that this implies that $\liminf_{k\to\infty} \|x^{k+1}-x^k\|>0$. 

Indeed, if this is false, then this assumption, as well as the fact that the sequence $(\|x^{k+1}-x^k\|_{k\in\N\cup\{0\}})$ is a sequence of nonnegative numbers, imply that there is an infinite set $N_1$ of $\{k_0,k_0+1,k_0+2,\ldots\}$ such that $\lim_{k\to\infty, k\in N_1}\|x^{k+1}-x^k\|=0$. Since $(x^k)_{k\in N_1}$ is contained in the bounded set $\Omega_{m+1}$, it has a convergent subsequence $(x^k)_{k\in N_2}$ (where $N_2$ is an infinite subset of $N_1$) whose limit, which we denote by $x^{\infty,2}$, is in the closed set $\Omega_{m+1}$. 

Knowing that $x^{k+1}=G_k(x^k)$ for every $k\in\N\cup\{0\}$ (see \beqref{eq:GBAP}), we conclude from Lemma \bref{lem:g_k(x^k)} and Lemma \bref{lem:UniformConvergence} (with $K:=\N\cup\{0\}$ there) that
\begin{equation*}
\lim_{k\to \infty, k\in N_2}\|x^{k+1}-x^k\|=\lim_{k\to \infty, k\in N_2}\|G_k(x^{k})-x^k\|=\|G_{\lambda_{\infty},\gamma_{\infty}}(x^{\infty,2})-x^{\infty,2}\|,
\end{equation*}
where $0<\lambda_{\infty}=\lim_{k\to\infty}\lambda_k$ and $0<\gamma_{\infty}=\lim_{k\to\infty}\gamma_k$ as we assume in the formulation of Theorem \bref{thm:Main} (since $\lambda_k\in (0,1]$ and $\gamma_k\in (0,1]$ for each $k\in\N\cup\{0\}$, we also have $\lambda_{\infty}\leq 1$ and $\gamma_{\infty}\leq 1$). Since $N_2\subseteq N_1$, we have $\lim_{k\to \infty, k\in N_2}\|x^{k+1}-x^k\|=\lim_{k\to \infty, k\in N_1}\|x^{k+1}-x^k\|=0$. Hence, $\|G_{\lambda_{\infty},\gamma_{\infty}}(x^{\infty,2})-x^{\infty,2}\|=0$, namely, $x^{\infty,2}$ is a fixed point of $G_{\lambda_{\infty},\gamma_{\infty}}$.

The operator $G_{\lambda_{\infty},\gamma_{\infty}}$ has a unique fixed point in $\Omega_{m+1}$ according to Lemma \bref{lem:|G(x)-G(y)|<|x-y|}, and this fixed point is $x^*$ according to Lemma \bref{lem:FixedPointMin}. This yields $x^{\infty,2}=x^*$ which is a contradiction because the fact that $x^k\notin B(x^*,\varepsilon)$ for all $k_0\leq k\in\N\cup\{0\}$ implies that $\|x^k-x^*\|\geq \varepsilon$ for all $k_0\leq k\in\N\cup\{0\}$, and hence $\varepsilon\leq \lim_{k\to\infty, k\in N_2}\|x^k-x^*\|=\|x^{\infty,2}-x^*\|$. Consequently, $\liminf_{k\to\infty} \|x^{k+1}-x^k\|>0$, as claimed, and hence there is a nonnegative integer $k_1\geq k_0$ such that $\|x^{k+1}-x^k\|\geq \liminf_{\ell\to\infty} \|x^{\ell+1}-x^{\ell}\|>0$ for all $k_1\leq k\in\N\cup\{0\}$. 

Denote $r:=\inf\{\|x^{k+1}-x^k\|\,|\, k_1\leq k\in \N\cup\{0\}\}$. According to what we showed in the previous paragraph $r>0$. Additionally, the definition of $r$ implies that there is an infinite subset $N_3$ of $\{k_1,k_1+1,k_1+2,k_1+3,\ldots\}$ such that $r=\lim_{k\to\infty, k\in N_3} \|x^{k+1}-x^k\|$. The sequence $(x^k)_{k\in N_3}$ is contained in the compact set $\Omega_{m+1}$ and hence has an accumulation point $x^{\infty,3}\in \Omega_{m+1}$, namely $x^{\infty,3}=\lim_{k\to\infty, k\in N_4}x^k$ for some infinite subset $N_4$ of $N_3$. Lemma \bref{lem:g_k(x^k)} and Lemma \bref{lem:UniformConvergence} imply that 
\begin{equation}\label{r=lim}
r=\lim_{k\to\infty, k\in N_4}\|x^{k+1}-x^k\|=\lim_{k\to\infty, k\in N_4}\|G_k(x^k)-x^k\|=\|G_{\lambda_{\infty},\gamma_{\infty}}(x^{\infty,3})-x^{\infty,3}\|. 
\end{equation}
Since $r>0$, we conclude that $x^{\infty,3}$ is not a fixed point of $G_{\lambda_{\infty},\gamma_{\infty}}$. In addition, since $G_{\lambda_{\infty},\gamma_{\infty}}(x^{\infty,3})=\lim_{k\to\infty, k\in N_4}G_k(x^k)$ as follows from Lemma \bref{lem:g_k(x^k)} and Lemma \bref{lem:UniformConvergence} (with $K:=\N\cup\{0\}$), another application of these lemmata (again with $K:=\N\cup\{0\}$) implies that 
\begin{equation*}
\|x^{k+2}-x^{k+1}\|=\|G_{k+1}(G_{k}(x^k))-G_{k}(x^k)\|\xrightarrow[k\to\infty, k\in N_4]{}\|G_{\lambda_{\infty},\gamma_{\infty}}^2(x^{\infty,3})-G_{\lambda_{\infty},\gamma_{\infty}}(x^{\infty,3})\|.
\end{equation*} 

On the one hand, the definition of $r$ implies that $r\leq \|x^{k+2}-x^{k+1}\|$ for all $k_1\leq k\in\N\cup\{0\}$, and, in particular, $r\leq \|G_{\lambda_{\infty},\gamma_{\infty}}^2(x^{\infty,3})-G_{\lambda_{\infty},\gamma_{\infty}}(x^{\infty,3})\|$. On the other hand the fact that $G_{\lambda_{\infty},\gamma_{\infty}}$ is Edelstein-contractive on $\Omega_{m+1}$ according to Lemma \bref{lem:|G(x)-G(y)|<|x-y|}, together with \beqref{r=lim} and the facts that $G_{\lambda_{\infty},\gamma_{\infty}}(x^{\infty,3})\neq x^{\infty,3}$, $x^{\infty,3}\in \Omega_{m+1}$ and $G_{\lambda_{\infty},\gamma_{\infty}}(x^{\infty,3})\in \Omega_{m+1}$ (since Lemma \bref{lem:|G(x)-G(y)|<|x-y|} implies that $G_{\lambda_{\infty},\gamma_{\infty}}$ maps $\Omega_{m+1}$ into itself), all imply that $\|G_{\lambda_{\infty},\gamma_{\infty}}^2(x^{\infty,3})-G_{\lambda_{\infty},\gamma_{\infty}}(x^{\infty,3})\|<\|G_{\lambda_{\infty},\gamma_{\infty}}(x^{\infty,3})-x^{\infty,3}\|=r$. We conclude that $r\leq \|G_{\lambda_{\infty},\gamma_{\infty}}^2(x^{\infty,3})-G_{\lambda_{\infty},\gamma_{\infty}}(x^{\infty,3})\|<r$, a contradiction. 

The previous paragraph shows that the initial assumption that there does not exist a subsequence of $(x^k)_{k=0}^{\infty}$ which converges to $x^*$ cannot hold, namely there exists a subsequence of $(x^k)_{k=0}^{\infty}$ which converges to $x^*$. Thus, there exists an infinite subset $K$ of $\N\cup\{0\}$ such that $\lim_{k\to\infty, k\in K}x^k=x^*$, and therefore Lemma \bref{lem:g_k(x^k)} and Lemma \bref{lem:UniformConvergence}, as well as Lemma \bref{lem:FixedPointMin}, imply that 
\begin{equation*}
\lim_{k\to\infty, k\in K}\|x^{k+1}-x^k\|=\lim_{k\to\infty, k\in K}\|G_k(x^k)-x^k\|=\|G_{\lambda_{\infty},\gamma_{\infty}}(x^*)-x^*\|=0.
\end{equation*}
\end{proof}

Now we are ready to present the proof of Theorem \bref{thm:Main}.
\begin{proof}[Proof of Theorem \bref{thm:Main}]
Lemma \bref{lem:ExistenceUniqueness} addresses the issues of existence of a solution to \beqref{eq:MinProb} under the conditions mentioned in the formulation of Theorem \bref{thm:Main}, its uniqueness under further assumptions, as well as the fact that it is the unique fixed point of each of the functions $G_{\lambda,\gamma}:X\to X$, $\lambda\in (0,1]$, $\gamma\in (0,2)$ defined in \beqref{eq:G_lambda_gamma}. 

It remains to show that $(x^k)_{k=0}^{\infty}$ converges to the unique solution $x^*\in\Omega_{m+1}$ to \beqref{eq:MinProb} when the following assumptions are imposed: $x^0\in\Omega_{m+1}$, the space $X$ is finite-dimensional, there is some $\rho>0$ such that $\Omega_{m+1}\cup(\cup_{i=1}^m\cup_{j=1}^{m_i}\Omega_{i,j})\subseteq B[0,\rho]$, $\Omega_i$ is strictly convex and satisfies $\Omega_i\cap \Omega_{m+1}=\emptyset$ for each $i\in\{1,2,\ldots,m\}$, and $\lambda_{\infty}:=\lim_{k\to\infty}\lambda_k$ and $\gamma_{\infty}:=\lim_{k\to\infty}\gamma_k$ exist and are positive (and hence they satisfy $\lambda_{\infty}\leq 1$ and $\gamma_{\infty}\leq 1$ since $\lambda_k\in (0,1]$ and $\gamma_k\in (0,1]$ for every $k\in\N\cup\{0\}$). 

Since $\Omega_{m+1}$ is closed and bounded and $X$ is finite-dimensional, both $\Omega_{m+1}$ and $\Omega_{m+1}^2$ are compact metric spaces, where the metrics are induced by the norm. Now let $\varepsilon>0$ be arbitrary. Denote 
\begin{equation}\label{eq:A_{0.5epsilon}}
A_{0.5\varepsilon}:=\{(x,y)\in \Omega_{m+1}^2\,|\, \|x-y\|\geq 0.5\varepsilon\}.
\end{equation}
This is a closed set of the compact metric space $\Omega_{m+1}^2$ and hence a compact metric space. Therefore, for all $i\in\{1,2,\ldots,m\}$ the continuous function $f_i:A_{0.5\varepsilon}\to \R$, defined by $f_i(x,y):=\|x-y\|-\|P_{i}(x)-P_i(y)\|$ for all $(x,y)\in A_{0.5\varepsilon}$, attains a minimum over $A_{0.5\varepsilon}$ at some point $(\wt{x}(i),\wt{y}(i))$. 

Since $\Omega_i\cap \Omega_{m+1}=\emptyset$ and $\Omega_i$ is strictly convex for every $i\in \{1,2,\ldots,m\}$, Lemma \bref{lem:|P_S(x)-P_S(y)|<|x-y|} implies that $P_i$ is Edelstein-contractive on $\Omega_{m+1}$ for all $i\in \{1,2,\ldots,m\}$. Since $\|\wt{x}(i)-\wt{y}(i)\|\geq 0.5\varepsilon$ by the definition of $A_{0.5\varepsilon}$, it follows that $\wt{x}(i)\neq \wt{y}(i)$ and, hence, $\|P_i(\wt{x}(i))-P_i(\wt{y}(i))\|<\|\wt{x}(i)-\wt{y}(i)\|$. Therefore, 
\begin{equation}
0<\|\wt{x}(i)-\wt{y}(i)\|-\|P_i(\wt{x}(i))-P_i(\wt{y}(i))\|=f_i(\wt{x}(i),\wt{y}(i))\leq \|x-y\|-\|P_{i}(x)-P_i(y)\|
\end{equation}
for all $(x,y)\in A_{0.5\varepsilon}$, and hence $r_{\varepsilon}:=\min_{i\in\{1,2,\ldots,m\}}\{f_i(\wt{x}(i),\wt{y}(i)), 0.5\varepsilon\}$ satisfies both $r_{\varepsilon}\leq 0.5\varepsilon$ and
\begin{equation}\label{eq:r_epsilon}
0<r_{\varepsilon}\leq \|x-y\|-\|P_{i}(x)-P_i(y)\|,\quad \forall (x,y)\in A_{0.5\varepsilon}\,\,\textnormal{and}\,\,\forall i\in\{1,\ldots,m\}. 
\end{equation}

We know from Lemma \bref{lem:Lemma_27_CensorMansourReem2024jour} that for each $i\in\{1,2,\ldots,m\}$ the sequence $(\wh{P}_{i,k})_{k=0}^{\infty}$ converges uniformly, over $B[0,\rho]$, to the orthogonal projection $P_{i}$ onto $\Omega_i$. Hence, in particular, $(\wh{P}_{i,k})_{k=0}^{\infty}$ converges uniformly to $P_i$ over $\Omega_{m+1}$ and, therefore, there is $k_0\in \N\cup\{0\}$ such that $\|\wh{P}_{i,k}(x)-P_i(x)\|<r_{\varepsilon}$ for all $x\in \Omega_{m+1}$, all $i\in\{1,2,\ldots,m\}$ and all $k_0\leq k\in\N\cup\{0\}$. 

Lemma \bref{lem:subsequence_x^*} implies that there is an infinite subset $K$ of $\N\cup\{0\}$ such that $x^*=\lim_{k\to\infty, k\in K}x^k$. Thus there is $k_0\leq k_1\in K$ such that $x^{k_1}\in B(x^*,\varepsilon)$. We show by induction that $x^k\in B(x^*,\varepsilon)$ for all $k_1\leq k\in\N\cup\{0\}$. 

Indeed, for $k=k_1$ this is the induction hypothesis. Assume now that the assertion is true up to $k_1\leq k\in\N\cup\{0\}$. We need to show that it holds also for $k+1$. Due to $x^*=G_{\lambda_{\infty},\gamma_{\infty}}(x^*)=(1-\lambda_{\infty})x^*+\lambda_{\infty} P_{m+1}((1-\gamma_{\infty})x^*+\gamma_{\infty}\sum_{i=1}^mP_{i}(x^*))$ according to Lemma \bref{lem:FixedPointMin}, we conclude from \beqref{eq:GBAP}, simple arithmetic, the triangle inequality, the nonexpansivity of the orthogonal projections $P_{i}$ (for each $i\in\{1,2,\ldots,m+1\}$), from the fact that $(\lambda_{\infty},\gamma_{\infty})\in (0,1]^2$, from the fact that $\beta=(\beta_i)_{i=1}^m$ is a positive weight vector, from the fact that $x^*\in \Omega_{m+1}$, and from the fact that $x^k\in \Omega_{m+1}$ (according to Lemma \bref{lem:x^k_is_bounded}) and hence $\|\wh{P}_{i,k}(x^k)-P_i(x^k)\|<r_\varepsilon$ (since $k_0\leq k_1\leq k$), that 
\begin{equation}\label{eq:x^(k+1)-x^*}
\begin{split}
\|x^{k+1}-x^*\|&=\|((1-\lambda_{\infty})x^k+\lambda_{\infty} P_{m+1}((1-\gamma_{\infty})x^k+\gamma_{\infty}\sum_{i=1}^m\wh{P}_{i,k}(x^k))\\
&-((1-\lambda_{\infty})x^*+\lambda_{\infty} P_{m+1}((1-\gamma_{\infty})x^*+\gamma_{\infty}\sum_{i=1}^mP_{i}(x^*)))\|\\
&\leq (1-\lambda_{\infty})\|x^{k}-x^*\|+\lambda_{\infty}\|P_{m+1}((1-\gamma_{\infty})x^k+\gamma_{\infty}\sum_{i=1}^m\beta_i\wh{P}_{i,k}(x^k))\\
&-P_{m+1}((1-\gamma_{\infty})x^*+\gamma_{\infty}\sum_{i=1}^m\beta_iP_{i}(x^*))\|\\
&\leq (1-\lambda_{\infty})\varepsilon+\lambda_{\infty}((1-\gamma_{\infty})\|x^k-x^*\|+\gamma_{\infty}\sum_{i=1}^m\beta_i\|\wh{P}_{i,k}(x^k)-P_{i}(x^*)\|)\\
&\leq (1-\lambda_{\infty})\varepsilon+\lambda_{\infty}(1-\gamma_{\infty})\|x^k-x^*\|\\&
+\lambda_{\infty}\gamma_{\infty}\sum_{i=1}^m\beta_i (\|\wh{P}_{i,k}(x^k)-P_{i}(x^k)\|+\|P_i(x^k)-P_i(x^*)\|)\\
&\leq (1-\lambda_{\infty})\varepsilon+\lambda_{\infty}(1-\gamma_{\infty})\varepsilon+\lambda_{\infty}\gamma_{\infty}\sum_{i=1}^m\beta_i(r_{\varepsilon}+\|P_i(x^k)-P_i(x^*)\|).
\end{split}
\end{equation}
We claim that $r_{\varepsilon}+\|P_i(x^k)-P_i(x^*)\|<\varepsilon$ for all $i\in\{1,2,\ldots,m\}$. Indeed, if we are in the case $\|x^k-x^*\|<0.5\varepsilon$, then the nonexpansivity of $P_i$ (for all $i\in\{1,2,\ldots,m\}$) and the fact that $r_{\varepsilon}\leq 0.5\varepsilon$ imply that $r_{\varepsilon}+\|P_i(x^k)-P_i(x^*)\|\leq 0.5\varepsilon+\|x^k-x^*\|< 0.5\varepsilon+0.5\varepsilon=\varepsilon$. If we are in the case $\|x^k-x^*\|\geq 0.5\varepsilon$, then $(x^k,x^*)\in A_{0.5\varepsilon}$ (see \beqref{eq:A_{0.5epsilon}}). This fact, as well as the fact that $\|x^k-x^*\|<\varepsilon$ and the fact that $r_{\varepsilon}\leq \|x^k-x^*\|-\|P_i(x^k)-P_i(x^*)\|$ for all $i\in\{1,2,\ldots,m\}$ (see \beqref{eq:r_epsilon}), imply that $r_{\varepsilon}<\varepsilon-\|P_i(x^k)-P_i(x^*)\|$ for all $i\in\{1,2,\ldots,m\}$. Hence $r_{\varepsilon}+\|P_i(x^k)-P_i(x^*)\|< \varepsilon$ for all $i\in\{1,2,\ldots,m\}$ in this case too. 

Therefore, \beqref{eq:x^(k+1)-x^*}, as well as the facts that $(\beta_i)_{i=1}^m$ is a positive weight vector and both $\lambda_{\infty}$ and $\gamma_{\infty}$ are positive, imply that 
\begin{equation*}
\|x^{k+1}-x^*\|<(1-\lambda_{\infty})\varepsilon+\lambda_{\infty}(1-\gamma_{\infty})\varepsilon+\lambda_{\infty}\gamma_{\infty} \sum_{i=1}^m \beta_i \varepsilon=\varepsilon,
\end{equation*}
namely $x^{k+1}\in B(x^*,\varepsilon)$. We conclude by induction that $x^k\in B(x^*,\varepsilon)$ for all $k_1\leq k\in\N\cup\{0\}$, and since $\varepsilon$ was an arbitrary positive number it follows that $(x^k)_{k=0}^{\infty}$ converges to the unique solution $x^*$ of \beqref{eq:MinProb}, as required.
\end{proof}

\begin{remark}\label{rem:Better_Than_whP}
A careful examination of the proof of Theorem \bref{thm:Main} and the lemmata on which it is based, shows that for each $i\in \{1,2,\ldots,m\}$ if the sequence of operators $(\wh{P}_{i,k})_{k=0}^{\infty}$ which is defined in \beqref{eq:ApproxProjection} is replaced by any sequence of operators which converges uniformly to $P_i$ over $B[0,\rho]$, then the conclusion of Theorem \bref{thm:Main} still holds. 
\end{remark}

{\noindent \textbf{Acknowledgements}: } This work was supported by U.S. National Institutes of Health (NIH) Grant Number R01CA266467 and by the Cooperation Program in Cancer Research of the German Cancer Research Center (DKFZ) and Israel\textquoteright s Ministry of Innovation, Science and Technology (MOST).

\subsection*{Data availability} Not applicable.
 
\subsection*{Ethics/Ethical} Not applicable.

\subsection*{Clinical trial} Not applicable.

\subsection*{Competing Interests} The authors declare no competing financial or non-financial interests.

\bibliographystyle{acm}
\bibliography{biblio}
\end{document}